\documentclass[11pt]{amsart}

\usepackage{amscd,amssymb,amsopn,amsmath,amsthm,mathrsfs,graphics,amsfonts,enumerate,verbatim,calc
}
\usepackage{bbm}
\usepackage[all,cmtip]{xy}
\usepackage[all]{xy}
\usepackage{tikz}
\usepackage{lscape}
\usepackage{enumitem}
\usetikzlibrary{matrix,arrows,decorations.pathmorphing,cd}
\usepackage{hyperref}
\hypersetup{colorlinks=true,linkcolor={blue}}
\usepackage{cleveref}
\usepackage{url}

\usepackage{scalerel}
\usepackage{stackengine,wasysym}
\usetikzlibrary {positioning}
\usetikzlibrary{patterns}
\usetikzlibrary{calc}
\definecolor {processblue}{cmyk}{0.96,0,0,0}

\usepackage{subfigure}

\usepackage{comment} 

\usepackage{ dsfont }

\usepackage{color}

\usepackage[OT2,OT1]{fontenc}
\newcommand\cyr{%
\renewcommand\rmdefault{wncyr}%
\renewcommand\sfdefault{wncyss}%
\renewcommand\encodingdefault{OT2}%
\normalfont
\selectfont}
\DeclareTextFontCommand{\textcyr}{\cyr}

\usepackage{amssymb,amsmath}

\DeclareFontFamily{OT1}{rsfs}{}
\DeclareFontShape{OT1}{rsfs}{n}{it}{<-> rsfs10}{}
\DeclareMathAlphabet{\mathscr}{OT1}{rsfs}{n}{it}

\numberwithin{equation}{section}
\newtheorem{theorem}{Theorem}[section]
\newtheorem{lem}[theorem]{Lemma}
\newtheorem{cor}[theorem]{Corollary}

\newtheorem{question}{Question}

\newtheorem{conjecturbe}[theorem]{Conjecture}

\newtheorem{prop}[theorem]{Proposition}

\theoremstyle{definition}
\newtheorem{defn}[theorem]{Definition}
\theoremstyle{remark}
\newtheorem{remark}[theorem]{Remark}

\newtheorem{example}[theorem]{Example}

\newcommand{\Ass}{\operatorname{Ass}}

\newcommand{\im}{\operatorname{im}}
\renewcommand{\ker}{\operatorname{ker}}
\newcommand{\tr}{\operatorname{tr}}

\newcommand{\grade}{\operatorname{grade}}
\newcommand{\gr}{\operatorname{gr}}

\newcommand{\pd}{\operatorname{pd}}
\newcommand{\id}{\operatorname{id}}
\newcommand{\Ext}{\operatorname{Ext}}

\newcommand{\Tor}{\operatorname{Tor}}
\newcommand{\Hom}{\operatorname{Hom}}

\newcommand{\End}{\operatorname{End}}

\newcommand{\depth}{\operatorname{depth}}

\newcommand{\coker}{\operatorname{coker}}

\newcommand{\Soc}{\operatorname{Soc}}

\newcommand{\rank}{\ensuremath{\operatorname{rank}}}
\renewcommand{\mod}{\operatorname{mod}}

\newcommand{\Tr}{\operatorname{Tr}}

\newcommand{\p}{\mathfrak{p}}

\newcommand{\m}{\mathfrak{m}}

\newcommand{\Z}{\mathbb{Z}}

\newcommand{\w}{\omega_R}
\newcommand{\chara}{\operatorname{char}}
\renewcommand{\bar}{\overline}

\newcommand{\Ann}{\operatorname{Ann}}
\newcommand{\Deep}{\operatorname{Deep}}
\newcommand{\DF}{\operatorname{DF}}

\def\grade{{\rm grade}}

\newcommand{\JLL}[1]{{\color{blue}\sf #1}} 

\begin{document}

\title[Annihilators of (co)Homology and their Influence on the Trace Ideal]{Annihilators of (co)Homology and their Influence on the Trace Ideal}

\author[Lyle]{Justin Lyle}
\email[Justin Lyle]{jll0107@auburn.edu}
\urladdr{https://jlyle42.github.io/justinlyle/}
\address{Department of Mathematics and Statistics \\ 221 Parker Hall\\
Auburn University\\
Auburn, AL 36849}

\author[Maitra]{Sarasij Maitra}
\email[Sarasij Maitra]{maitra@math.utah.edu}
\urladdr{https://sarasij93.github.io}
\address{155 South 1400 East, JWB 233 \\ Salt Lake City, UT 84112}



\begin{abstract}
Let $(R,\m)$ be a commutative Noetherian local ring, and suppose $R$ is Cohen-Macaulay with canonical module $\w$. We develop new tools for analyzing questions involving annihilators of several homologically defined objects. Using these, we study a generalization introduced by Dao-Kobayashi-Takahashi of the famous Tachikawa conjecture, asking in particular whether the vanishing of $\m \Ext_R^i(\w,R)$ should force the trace ideal of $\w$ to contain $\m$, i.e., for $R$ to be nearly Gorenstein. We show this question has an affirmative answer for numerical semigroup rings of minimal multiplicity, but that the answer is negative in general. Our proofs involve a technical analysis of homogeneous ideals in a numerical semigroup ring, and exploit the behavior of Ulrich modules in this setting.
\end{abstract}

\maketitle

\section{Introduction}

Let $(R,\m,k)$ be a Noetherian local ring or a positively graded $k$-algebra with homogeneous maximal ideal $\m$. The following conjecture, first stated in its current version by \cite{AB05}, serves a commutative algebra version of the famous Tachikawa conjecture, which is intimately related to several longstanding questions in the representation theory of Artin algebras:
\begin{conjecturbe}\label{tachikawa}

Suppose $R$ is Cohen-Macaulay with canonical module $\w$. If $\Ext^i_R(\w,R)=0$ for all $i>0$, then $R$ is Gorenstein.

\end{conjecturbe}

Conjecture \ref{tachikawa} is known if $R$ is generically Gorenstein, i.e., $R_{\p}$ is Gorenstein for all $\p \in \Ass(R)$ \cite{AB05} (cf. \cite{HH05}), if $R$ is positively graded over $k$ \cite{Ze90}, or if certain invariants are small, e.g. if $\mu_R(\w) \le 2$ or if $\m^3=0$ \cite{HS04}.
Given the ubiquity and utility of the Gorenstein property, several axes have been introduced along which one can measure the failure of this property to hold. One such object that has received a great deal of recent attention is the trace ideal of the canonical module (see for instance \cite{HT19,DK21,He21,He22,HM22,He23} etc.). We recall this ideal is defined as
\[\tr_R(\w):=\langle f(x) \mid x \in \w, f \in \Hom_R(\w,R) \rangle .\]
The ideal $\tr_R(\w)$ gives a measure of how far away $R$ is from being Gorenstein. Indeed, $R$ is Gorenstein if and only if $\tr_R(\w)=R$, and it follows from \cite[Lemma 2.1]{HT19} that $\tr_R(\w)$ defines the non-Gorenstein locus of $R$. If $R$ is not Gorenstein, then the next best thing by this measure is if $\tr_R(\w)=\m$, and for this reason rings with $\tr_R(\w) \supseteq \m$ are called \emph{nearly Gorenstein}, and have been shown to have many desirable properties; see for example \cite{HT19,CS21,DK21}. 

The following question was posed in \cite[Question 4.3]{DK21} and serves as a variation of the Tachikawa conjecture for the nearly Gorenstein condition:

\begin{question}\label{daoetal}

Suppose $(R,\m,k)$ is a Cohen-Macaulay local ring with canonical module $\w$. Suppose $\m \Ext^i_R(\w,R)=0$ for all $i>0$. Must $R$ be nearly Gorenstein?

\end{question}



As Question \ref{daoetal} is an extension of the Tachikawa conjecture, it is natural to narrow its focus to cases where the Tachikawa conjecture is known to hold. Some evidence in this setting was recently provided by Dao-Kobayashi-Takahashi, who gave an affirmative answer for generically Gorenstein rings of type $2$ \cite[Theorem 4.4]{DK21}, and Esentepe, who established the case where $R$ is generically Gorenstein and $\tr_R(\w)$ is radical \cite[Theorem C]{Es24}.



In this work, we develop a variety of new tools for attacking problems involving (co)homological objects associated to ideals and their annihilators. Our methods allow us to prove several new facts about numerical semigroup rings, and especially for numerical semigroup rings of minimal multiplicity, allow to describe a tight relationship between good properties of trace and colon ideals, and direct sum decompositions of the first syzygy of an ideal. Our main result illustrating the utility of our techniques is the following (see Theorem \ref{generalideals}):


\begin{theorem}\label{introthm1}

Let $R$ be a numerical semigroup ring of minimal multiplicity and let $I=(x_1,\dots,x_n)$ be a homogeneous ideal of $R$ with $|x_i|<|x_{i+1}|$ for all $i$. Let $\w$ be the canonical module of $R$. Then the following are equivalent:

\begin{enumerate}

\item[$(1)$] $\m \Ext^1_R(I,\Hom_R(I,\w))=0$.
\item[$(2)$] $\m \Tor^R_1(I,R/I)=0$.
\item[$(3)$] $\m \bigwedge^2_R(I)=0$.
\item[$(4)$] $\m B_1(I) \subseteq IZ_1(I)$. 
\item[$(5)$] $\tr_R(I)=\m$.
\item[$(6)$] $(x_1:I)=\m$.
\item[$(7)$] $\Omega^1_R(I) \cong \bigoplus_{i=2}^n \m(-|x_i|)$ as graded $R$-modules.

\end{enumerate}
Moreover, in each of the conditions $(1)-(4)$ we may replace $\m$ by its minimal homogeneous reduction. 
\end{theorem}
Applying Theorem \ref{introthm1}, we show Question \ref{daoetal} has an affirmative answer in this setting (see Theorem \ref{nearlygormainthm}). This behavior turns out to be quite subtle, as we provide examples showing Question \ref{daoetal} has a negative answer in general, and in particular, when the hypotheses of Theorem \ref{introthm1} are relaxed only slightly (see Theorem \ref{finex}):

\begin{theorem}\label{introthm2}
There exists a numerical semigroup ring $R$ with embedding dimension $4$, multiplicity $5$, and type $3$ for which Question \ref{daoetal} has a negative answer.
\end{theorem}

In particular, such an example from Theorem \ref{introthm2} has almost minimal multiplicity. Moreover, there exists such an example that is a far-flung Gorenstein ring in the sense of \cite{He23}, and so may be viewed as being as far from nearly Gorenstein as possible. As a byproduct, our example gives a situation where the equality of \cite[Question 1.4]{Es24} does not hold. 

As an additional application of Theorem \ref{introthm1}, we have the following that captures a new fundamental structure of numerical semigroup rings of minimal multiplicity (see Corollary \ref{thisiscool}):
\begin{theorem}\label{thisiscoolintro}
Suppose $R$ is a numerical semigroup ring that is not a DVR. Then $R$ has minimal multiplicity if and only if $Z_1(\m) \cong \m^{\oplus e(R)-1}$. 
\end{theorem}

We now briefly outline the structure of our paper. In Section \ref{background}, we provide some background and set notation used throughout the paper. In Section \ref{prelim}, we provide new tools ranging across various topics that are instrumental in the proofs of our main theorems: Section \ref{tracesection} overviews several different ways to calculate trace ideals, Section \ref{ulrichsection} provides several new results on Ulrich modules and also provides direct extensions of the work of \cite{DE21}, and Section \ref{numericalsemigroupsection} gives new results on syzygies and presentation matrices of homogeneous ideals in a numerical semigroup ring. Section \ref{main} contains the proofs of our main results: Theorems \ref{generalideals} and \ref{nearlygormainthm}. Finally, in Section \ref{counterex}, we prove Theorem \ref{introthm2} by providing an explicit extended example and comment on its minimality.

\section{Background and Notation}\label{background}

Throughout we let $(R,\m,k)$ be a Noetherian local ring or a positively graded algebra of dimension $d$ over the field $k$ with homogeneous maximal ideal $\m$. All modules are assumed to be finitely generated, and in the graded case are assumed to be homogeneous. For an $R$-module $M$, we write $\mu_R(M)$ for its minimal number of generators, and we write $\Omega^i_R(M)$ for the $i$th syzygy of $M$, taking $\Omega^0_R(M)=M$. We write $\beta^R_i(M):=\mu_R(\Omega^i_R(M))=\dim_k(\Tor^R_i(M,k))$, and we let $l_R(M)$ denote the length of the $R$-module $M$. We write $r(R):=\dim_k \Ext^{\depth(R)}_R(k,R)$ for the type of $R$. If $I$ is an $\m$-primary ideal in $R$, we let $e_R(I,M)$ denote the Hilbert-Samuel multiplicity of $M$ with respect to $I$, so $e_R(I,M):=\displaystyle \lim_{n \to \infty} \dfrac{d! l_R(M/I^nM)}{n^{d}}$. In the case $I=\m$, we simply write $e_R(M):=e_R(\m,M)$ and we frequently will drop the subscript $R$ for the case when $M=R$. In the graded setting, we let $H_M(t):=\sum_{i \in \Z} (\dim_k M_i)t^i$ denote the Hilbert series of the graded $R$-module $M$. We recall that when $R$ is standard graded, we may write $H_M(t)$ uniquely in the form $H_M(t)=\dfrac{\epsilon_M(t)}{(1-t)^{\dim(M)}}$ for $\epsilon_M(t) \in \Z[t,t^{-1}]$ with $\epsilon_M(1) \ne 0$, and that $\epsilon_M(1)=e_R(M)$ when $\dim(M)=d$ (see \cite[Corollary 4.1.8]{BH93}).

When $R$ is Cohen-Macaulay, we always have the so-called Abhyankar bound $e(R) \ge \mu_R(\m)-d+1$, and rings for which equality is obtained are said to have \emph{minimal multiplicity} \cite[(1)]{Ab67} (cf. \cite[Theorem 1.2]{Sa76}). We let $(-)^*:=\Hom_R(-,R)$, and if $R^{\oplus m} \xrightarrow{A} R^{\oplus n} \rightarrow M \rightarrow 0$ for $M$, we let $\Tr_R(M):=\coker(A^T)$, so there is an exact sequence of the form
\[0 \rightarrow M^* \rightarrow R^{\oplus n} \xrightarrow{A^T} R^{\oplus m} \rightarrow \Tr_R(M) \rightarrow 0.\]
\noindent When $R$ is Cohen-Macaulay with canonical module $\omega_R$, we write $(-)^\vee:=\Hom_R(-,\omega_R)$.
We let $\bar{R}$ denote the integral closure of $R$ and we let $\mathfrak{c}_R$ denote the conductor ideal of $R$, so $\mathfrak{c}_R:=(R:_{\bar{R}} \bar{R})$.


If $I$ is an ideal in $R$ with minimal generating set $(x_1,\dots,x_{\mu_R(I)})$, we let $H_i(I)$, $B_i(I)$, and $Z_i(I)$ denote the $i$th Koszul homology, boundary, and cycle respectively with respect to this generating set. We rely on context to dictate either a specific generating set that has been chosen or else that the choice of generating set is immaterial to the problem at hand. Note that it is well-known that $IH_i(I)=0$ for any $i$, so we always have $IZ_i(I) \subseteq B_i(I)$.  We write $\delta_1(I)$ for the so-called \emph{syzygetic} part of $I$, which we define to be $\delta_1(I):=Z_1(I) \cap IR^{\oplus \mu_R(I)}/B_1(I)$. It is well-known that $\delta_1(I)$ is isomorphic to the kernel $K$ of the natural surjection $S^2_R(I) \to I^2$, and that there is an exact sequence
\begin{equation*}\label{syzygeticexact}
0 \to \delta_1(I) \to H_1(I) \to (R/I)^{\oplus \mu_R(I)} \to I/I^2 \to 0,
\end{equation*}
see \cite{SV81} (cf. \cite[Paragraph after Proposition 5.19]{Va05}). 
We will make heavy use of the following:
\begin{prop}\label{thetorthing}
If $I$ is an ideal of positive grade in $R$, then there is a natural exact sequence
\[\mbox{$\bigwedge^2_R(I)$} \xrightarrow{i} \Tor^R_1(I,R/I) \rightarrow K \rightarrow 0\]
fitting into a commutative diagram:
\[\begin{tikzcd}
	& {\bigwedge^2_R(I)} & {\Tor^R_1(I,R/I)} & {K} & 0 \\
	0 & {B_1(I)/IZ_1(I)} & {\dfrac{Z_1(I) \cap IR^{\oplus \mu_R(I)}}{IZ_1(I)}} & {\delta_1(I)} & 0
	\arrow["i", from=1-2, to=1-3]
	\arrow["\alpha", from=1-2, to=2-2]
	\arrow[from=1-3, to=1-4]
	\arrow["\cong" {anchor=south, rotate=-90}, from=1-3, to=2-3]
	\arrow[from=1-4, to=1-5]
	\arrow["\cong" {anchor=south, rotate=-90},from=1-4, to=2-4]
	\arrow[from=2-1, to=2-2]
	\arrow[from=2-2, to=2-3]
	\arrow[from=2-3, to=2-4]
	\arrow[from=2-4, to=2-5]
\end{tikzcd}\]
Moreover, we have the following:
\begin{enumerate}
\item[$(1)$] If $\m \bigwedge^2_R(I)=0$ or $\chara k \ne 2$, then $i$ is injective and $\alpha$ is an isomorphism.
\item[$(2)$] If $\chara k \ne 2$, then the top and bottom rows are split exact.   
\item[$(3)$] If $\m \Tor^R_1(I,R/I)=0$, then there is a split exact sequence of the form
\[0 \to \textstyle{\bigwedge^2_R}(I) \otimes_R k \to \Tor^R_1(I,R/I) \to \delta_1(I) \to 0.\]
In particular, $\m \delta_1(I)=0$ and we have $\dim_k(\Tor^R_1(I,R/I))=\displaystyle {\mu_R(I) \choose 2}+\dim_k(\delta_1(I))$.
\end{enumerate}
\end{prop}

\begin{proof}

We note one may identify $\Tor^R_1(I,R/I)$ with the kernel of the natural surjection $p:I \otimes_R I \to I^2$. For any $R$-module $M$, we let $\iota_M:\bigwedge^2_R(M) \to M \otimes_R M$ be the natural \emph{antisymmetrization map} given on elementary wedges by $x \wedge y \mapsto x \otimes y-y \otimes x$. The map $i$ is the restriction of $\iota_I$, noting that $\iota_I$ maps into the kernel of $p$, and existence of the diagram follows from \cite[Proposition 1.1 and Corollary 1.2]{SV81}.

For claims $(1)-(2)$, first suppose $\m \bigwedge^2_R(I)=0$. Then there is a commutative diagram:
\[\begin{tikzcd}[cramped]
	{\bigwedge^2_R(I)} & {I \otimes_R I} \\
	{\bigwedge^2_k(I/\m I)} & {I/\m I \otimes_k I/\m I}
	\arrow["{\iota_I}", from=1-1, to=1-2]
	\arrow["\cong" {anchor=south, rotate=-90},from=1-1, to=2-1]
	\arrow[from=1-2, to=2-2]
	\arrow["{\iota_{I/\m I}}", from=2-1, to=2-2]
\end{tikzcd}\]
whose vertical arrows are the natural maps. Over a field, it is well-known that $\iota_M$ is injective for every $M$ (see e.g. \cite[Lemma 5.9]{CC23}), so in particular $\iota_{I/\m I}$ is injective. Commutativity of the diagram then forces $\iota_I$ to be injective. But as $i$ is a restriction of $\iota_I$, it is injective as well. That $\alpha$ is an isomorphism now follows from the snake lemma. 

If instead, $\chara k \ne 2$, then $i$ is split injective by \cite[Remark after Proposition 1.1]{SV81}, and the snake lemma again forces $\alpha$ to be an isomorphism. But the the diagram gives a chain isomorphism, so the bottom row splits as well. 

For claim $(3)$, as $\im(\iota_I) \hookrightarrow \Tor^R_1(I,R/I)$, we have $\m \im(\iota_I)=0$, and it suffices to show $\mu_R(\bigwedge^2_R(I))=\mu_R(\iota_I)$, equivalently that no minimal generator of $\bigwedge^2_R(I)$ is contained in $\ker(\iota_I)$. If $x_1,\dots,x_n$ is a minimal generating set for $I$, then $\{x_i \wedge x_j \mid i<j\}$ is a minimal generating set for $\bigwedge^2_R(I)$. But $\{x_i \otimes x_j \mid 1 \le i,j \le n\}$ is a minimal generating set for $I \otimes_R I$, so $\iota_M(x_i \wedge x_j)=x_i \otimes x_j-x_j \otimes x_i \ne 0$ if $i<j$. Thus $\mu_R(\im(\iota_M))=\mu_R(\bigwedge^2_R(I))$, completing the proof.
\end{proof}

It is standard practice in this theory to identify $K$ and $\delta_1(I)$ through the isomorphism discussed in Proposition \ref{thetorthing}. Henceforth, we will follow this convention as well.

\section{Preliminary Results}\label{prelim}

In this section we provide results of independent interest on several different topics that will be instrumental in the proofs of our main theorems.

\subsection{Trace Ideals}\label{tracesection}

In this section we recall the notion of trace ideals and several facts we will need in later sections.

\begin{defn}\label{tracedef}

Consider the map $\theta_R^M:M^* \otimes_R M \to R$ given by $f \otimes x \mapsto f(x)$. The image of $\theta$ is an ideal of $R$ called the \emph{trace ideal} of $M$, denoted $\tr_R(M)$. 

\end{defn}

For ease of reference, we summarize some well-known properties of trace ideals each of which will be used either explicitly or implicitly throughout.

\begin{prop}[{\cite[Proposition 2.8]{Li17}}]\label{basictraceprop}
For any $R$, any $R$-modules $M,N$, and any ideal $R$, we have the following:
\begin{enumerate}

\item[$(1)$] $\tr_R(M \oplus N)=\tr_R(M)+\tr_R(N)$.
\item[$(2)$] $\tr_R(M)=R$ if and only if $R$ is a direct summand of $M$.
\item[$(3)$] $I \subseteq \tr_R(I)$ with equality if and only if $I=\tr_R(M)$ for some $R$-module $M$.
\item[$(4)$] If $S$ is a flat $R$-algebra, then $\tr_S(M \otimes_R S)=\tr_R(M)S$. In particular, taking trace ideals respects localization and completion.

\end{enumerate}
\end{prop}

Our next Proposition gives several ways in which the trace ideal of $M$ can be calculated. Prior to presenting it, we will need the following elementary lemma.

\begin{lem}\label{flip}

Suppose $I$ is an ideal of $R$ and that $x,y \in I$ are nonzerodivisors on $R$. Then $x(y:I)=y(x:I)$. 

\end{lem}

\begin{proof}

By symmetry, it suffices to show that $x(y:I) \subseteq y(x:I)$. Pick $a \in (y:I)$. Then as $y$ is a nonzerodivsor, there is a unique $a' \in R$ with $ax=a'y$. We claim that $a' \in (x:I)$, from which the claim will follow. Indeed, if $b \in I$, then again as $y$ is a nonzerodivisor there is a unique $b'$ with $ab=b'y$ and we have
\[axb=a'yb \Rightarrow b'xy=a'yb \Rightarrow b'x=a'y.\]
It follows that $a' \in (x:I)$, completing the proof. 
\end{proof}


\begin{prop}\label{traceprop}

The following statements hold:

\begin{enumerate}

\item[$(1)$] Let $A$ be a minimal presentation matrix for $M$ and let $B$ be a presentation matrix for $\im(A^T)$. Then $\tr_R(M)=I_1(B)$, where the latter is the ideal generated by the entries of $B$.

\item[$(2)$] Let $I$ be an ideal of positive grade and let $x \in I$ be a nonzerodivisor on $R$. Then $\tr_R(I)=(I(x:I):x)$. 

\item[$(3)$] Let $I$ be an ideal of positive grade and let $x_1,\dots,x_n$ be a minimal generating set for $I$ where each $x_i$ is a nonzerodivisor on $R$ (such a generating set exists by prime avoidance). Then $\tr_R(I)=\sum_{i=1}^n (x_i:I)$. 

\item[$(4)$] If $I$ is an ideal of positive grade, then $\tr_R(I)$ is an $\End_R(I)$-ideal and $\Tor^R_1(\Tr_R(I),I) \cong \End_R(I)/\tr_R(I)$. In particular, $\tr_R(I)=\Ann_R(\Tor^R_1(\Tr_R(I),I))$
\end{enumerate}

\end{prop}

\begin{proof}

Items $(1)$ and $(2)$ follow from \cite[Remark 2.5]{Li17} and \cite[Proposition 2.4]{KT19-2} respectively. For $(3)$, we have from item $(2)$, that $x_1 \tr_R(I)=(x_1:I)I$, and we observe that 
\[(x_1:I)I=(x_1:I)\sum_{i=1}^n (x_i)=\sum_{i=1}^n x_i(x_1:I).\]
By Lemma \ref{flip} this equals $\sum_{i=1}^n x_1(x_i:I)=x_1 \sum_{i=1}^n (x_i:I)$. So $x_1 \tr_R(I)=x_1 \sum_{i=1}^n (x_i:I)$. As $x_1$ is a nonzerodivisor, we have $\tr_R(I)=\sum_{i=1}^n (x_i:I)$, establishing $(3)$. For $(4)$, it is well-known (see \cite[Exercise 4.3.1]{LW12}) that $\End_R(I)$ is commutative and that there is an exact sequence 
\[I^* \otimes_R I \xrightarrow{\theta} \End_R(I) \rightarrow \Tor^R_1(\Tr_R(I),I) \rightarrow 0\]
of $\End_R(I)$-modules where $\theta$ is given on elementary tensors by $f \otimes x \mapsto \alpha_{f,x}$ with $\alpha_{f,x}(y)=f(y)x=f(x)y$ (see \cite[Proposition 12.9]{LW12}. In other words, $\alpha_{f,x}$ is the multiplication map by $f(x)$ on $I$ in this setting. Defining $\beta:\tr_R(I) \to \End_R(I)$ by sending $x$ to its corresponding multiplication map, we see that $\beta$ restricts to an isomorphism $\tr_R(I) \to \im(\theta)$, which identifies $\tr_R(I)$ as an ideal of $\End_R(I)$, and shows $\Tor^R_1(\Tr_R(I),R) \cong \End_R(I)/\tr_R(I)$ under this identification. In particular, $r \in R$ annihilates $\Tor^R_1(\Tr_R(I),I)$ if and only if $r \id_I \subseteq \tr_R(I)$, i.e., if $r \in \tr_R(I)$, completing the proof.

\end{proof}

\subsection{Ulrich Modules}\label{ulrichsection}

The following has been a key object of study in recent years.

\begin{defn}

An $R$-module $M$ is said to be \emph{Ulrich} if it is maximal Cohen-Macaulay and $\mu_R(M)=e_R(M)$. 

\end{defn}

\begin{remark}

After extending the residue field, if needed, a maximal Cohen-Macaulay module $M$ is Ulrich if and only if $\underline{x}M=\m M$ for a minimal reduction $\underline{x}$ of $\m$.

\end{remark}

In \cite{Ul84}, Ulrich asked whether every Cohen-Macaulay local ring has an Ulrich module. Recent work of \cite{IM24} shows that $R$ need not admit an Ulrich module even if $R$ is a complete intersection domain of dimension $2$ or a Gorenstein normal domain of dimension $2$. However, there are several cases where Ulrich modules are known to exist, and they are a powerful tool to have access to when they do. Two cases where they are known to exist abundantly are when $R$ is a Cohen-Macaulay local ring of dimension $1$, where any sufficiently high power of $\m$ is Ulrich \cite[Theorem 2.5]{SV74}, and the case where $R$ is Cohen-Macaulay and has minimal multiplicity, where $\Omega^1_R(M)$ is Ulrich for any maximal Cohen-Macaulay $R$-module $M$ provided $\Omega^1_R(M)$ has no free summand \cite[Theorem B]{KT19}. For our purposes we will need to develop the properties of Ulrich modules in these settings a bit further than what is known in the literature. Before we present our results in this direction, we will need to set some notation and provide several preparatory results of independent interest. These results are aimed at controlling when certain modules can have nonzero free summands, in light of the above.

We let $\mod(R)$ be the category of finitely generated $R$-modules. We recall several full subcategories of $\mod(R)$ introduced by \cite{DE21}: the category $\Deep(R)$ consists of $M \in \mod(R)$ with $\depth(M) \ge \depth(R)$, $\Omega \Deep(R)$ consists of $M \in \mod(R)$ fitting into an exact sequence $0 \to M \to F \to X \to 0$ where $F$ is free and $X \in \Deep(R)$, and $\DF(R)$ (standing for deeply faithful) consists of $M \in \mod(R)$ for which there is an exact sequence $0 \to R \to M^{\oplus n} \to X \to 0$ with $X \in \Deep(R)$. 



\begin{lem}\label{notminsyz}
Suppose $M \in \Omega \Deep(R)$ so that there is an exact sequence $0 \rightarrow M \xrightarrow{i} R^{\oplus n} \rightarrow X \rightarrow 0$ for some $n$ and $X \in \Deep(R)$. If $M \in \DF(R)$, then $M$ is not a minimal syzygy of $X$, i.e., $\im(i) \nsubseteq \m R^{\oplus n}$.  
\end{lem}

\begin{proof}

As $M \in \Omega \Deep(R) \cap \DF(R)$, it follows from \cite[Lemma 2.10]{DE21} that $R$ is a summand of $M$, so there is a split exact sequence $0 \rightarrow R \xrightarrow{j} M \rightarrow N \rightarrow 0$. In particular, we note that $N \in \Deep(R)$ since $M$ is. We have a commutative diagram of the form:
\[\begin{tikzcd}
	0 & R & M & N & 0 \\
	0 & R & {R^{\oplus n}} & {C} & 0
	\arrow[from=1-1, to=1-2]
	\arrow["j", from=1-2, to=1-3]
	\arrow[equals, from=1-2, to=2-2]
	\arrow[from=1-3, to=1-4]
	\arrow["i", from=1-3, to=2-3]
	\arrow[from=1-4, to=1-5]
	\arrow[dashed, from=1-4, to=2-4]
	\arrow[from=2-1, to=2-2]
	\arrow["{i \circ j}", from=2-2, to=2-3]
	\arrow[from=2-3, to=2-4]
	\arrow[from=2-4, to=2-5]
\end{tikzcd}\]
where $C:=\coker(i \circ j)$.

Applying the snake lemma to this diagram, we obtain an exact sequence of the form
\[0 \to N \to C \to X \to 0.\]
As $N,X \in \Deep(R)$, it follows from the depth lemma that $C \in \Deep(R)$. But we evidently have $\pd_R(C) \le 1$, so the Auslander-Buchsbaum formula forces $C$ to be free. It follows that $i \circ j$ is a split injection, so in particular $\im(i)$ must contain a minimal generator of $R^{\oplus n}$, i.e., $\im(i) \nsubseteq \m R^{\oplus n}$, so $M$ is not a minimal syzygy of $X$.  
\end{proof}

\begin{theorem}\label{dualsummand}

Suppose $R$ has depth $t$. If $\Ext^i_R(M,R)=0$ for all $1 \le i \le t-1$ and $M^* \in \DF(R)$, then $R$ is a summand of $M$.

\end{theorem}

\begin{proof}

Let 
\[\cdots \rightarrow R^{\oplus \beta^R_{i+1}(M)} \xrightarrow{\partial_{i}} R^{\oplus \beta^R_{i}(M)} \xrightarrow{\partial_{i-1}} \cdots \xrightarrow{\partial_1} R^{\oplus \beta^R_1(M)} \xrightarrow{\partial_0} R^{\oplus \beta^R_0(M)} \rightarrow M \rightarrow 0\]
be a minimal free resolution of $M$. Applying $(-)^*$, we have, since $\Ext^i_R(M,R)=0$ for all $1 \le i \le t-1$, an exact sequence of the form
\[0 \rightarrow M^* \xrightarrow{i} 
 R^{\oplus \beta^R_0(M)} \xrightarrow{\partial_0^T} R^{\oplus \beta^R_1(M)} \xrightarrow{\partial^T_1} \cdots \xrightarrow{\partial^T_{t-1}} R^{\oplus \beta^R_t(M)} \to C \to 0.\]
Applying the depth lemma to this exact sequence shows that $\im(\partial^T_0) \in \Deep(R)$, and it follows from Lemma \ref{notminsyz} that $M^*$ is not a minimal syzygy of $\im(\partial^T_0)$. Letting $R^{\mu_R(M^*)} \twoheadrightarrow M^*$ be a free module surjecting onto $M^*$, it follows there is an exact sequence
\[R^{\oplus \mu_R(M^*)} \xrightarrow{B} R^{\oplus \beta^R_0(M)} \xrightarrow{\partial^T_0} R^{\oplus \beta^R_1(M)} \rightarrow \Tr_R(M) \rightarrow 0\]
where the matrix $B$ contains a unit. Then there are invertible matrices $P$ and $Q$ for which $QBP^{-1}$ has the block form $\begin{pmatrix} 1 & 0 \\ 0 & B' \end{pmatrix}$. This gives rise to a chain isomorphism:
\[\begin{tikzcd}
	{R^{\oplus \mu_R(M^*)}} & {R^{\oplus \beta^R_0(M)}} & {R^{\oplus \beta^R_1(M)}} & {\Tr_R(M)} & 0 \\
	{R^{\oplus \mu_R(M^*)}} & {R^{\oplus \beta^R_0(M)}} & {R^{\oplus \beta^R_1(M)}} & {\Tr_R(M)} & 0
	\arrow["B", from=1-1, to=1-2]
	\arrow["P", from=1-1, to=2-1]
	\arrow["{\partial^T_0}", from=1-2, to=1-3]
	\arrow["Q", from=1-2, to=2-2]
	\arrow[from=1-3, to=1-4]
	\arrow[equals, from=1-3, to=2-3]
	\arrow[from=1-4, to=1-5]
	\arrow[equals, from=1-4, to=2-4]
	\arrow["{QBP^{-1}}", from=2-1, to=2-2]
	\arrow["{\partial^T_0 Q^{-1}}", from=2-2, to=2-3]
	\arrow[from=2-3, to=2-4]
	\arrow[from=2-4, to=2-5]
\end{tikzcd}\]
Since $QBP^{-1}$ has the block form $\begin{pmatrix} 1 & 0 \\ 0 & B' \end{pmatrix}$ and since $\partial^T_0 Q^{-1} \cdot QBP^{-1}=0$, it must be that $\partial^T_0 Q^{-1}$ has a column of $0$'s. Then $(\partial^T_0 Q^{-1})^T=(Q^T)^{-1} \partial_0$ has a row of all $0$'s. But this implies $R$ is a summand of $\coker((Q^T)^{-1} \partial_0) \cong \coker(\partial_0)=M$, as desired.
\end{proof}

\begin{theorem}\label{freesummand}
Suppose $R$ has depth $t$ and suppose $\Ext^i_R(M,R)=0$ for all $1 \le i \le t-1$. If $F$ is a free $R$-module, then $F$ is a summand of $M^*$ if and only if $F$ is a summand of $M$. 
\end{theorem}

\begin{proof}
By induction on the rank of $F$, it suffices to show that if $R$ is a summand of $M^*$, then $R$ is also a summand of $M$. Since $\Ext^i_R(M,R)=0$ for all $1 \le i \le t-1$, it follows from \cite[Lemma 3.1]{DE21} that $M^* \in \Deep(R)$. If $M^* \cong R \oplus N$ for some $N$, then $N \in \Deep(R)$, and from the split exact sequence $0 \to R \to M \to N \to 0$, we see that $M^* \in \DF(R)$. Theorem \ref{dualsummand} then forces $R$ to be a summand of $M$, as desired.
\end{proof}

Theorem \ref{freesummand} gives a significant sharpening of \cite[Lemmas 3.2 and 3.3]{DE21} and leads, as an immediate consequence, to the following extension of \cite[Lemma 3.9]{DE21} and \cite[Lemma 2.13]{CG19}.
\begin{cor}\label{dualsummandweak}

Suppose $\depth(R) \le 1$. Then a free $R$-module $F$ is a summand of $M^*$ if and only if $F$ is a summand of $M$.

\end{cor}


It is shown in \cite[Lemma 3.3]{DE21} that the vanishing of $\Ext^i_R(M,R)$ for all $1 \le i \le \depth(R)$ implies that $M^*$ is free if and only if $M$ is free. It is noted in \cite[Example 3.5]{DE21} that the range on the vanishing of $\Ext$ in this result cannot be improved to $1 \le i \le \depth(R)-1$, as one can take for instance $M=T \oplus F$ where $T$ has finite length. Our approach shows that this type of example is essentially the only one that can exist.

\begin{cor}\label{finalsummandcor}

Suppose $R$ has depth $t$ and that $\Ext^i_R(M,R)=0$ for all $1 \le i \le t-1$. If $M^*$ is free, then $M \cong T \oplus M^*$ where $T$ is either $0$ or has grade $t$. If additionally $R$ is Cohen-Macaulay, then $T$ has finite length.

\end{cor}

\begin{proof}

As $M^*$ is free, Theorem \ref{freesummand} forces $M^*$ to be a summand of $M$. We may thus write $M \cong T \oplus M^*$. But using again that $M^*$ is free, we have $M^* \cong T^* \oplus M^{**} \cong T^* \oplus M^*$ and it follows that $T^*=0$. By assumption, we then have $\Ext^i_R(T,R)=0$ for all $0 \le i \le t-1$. It follows that either $T=0$ or $\grade(T)=t$, completing the proof.
\end{proof}

In order to apply the above results, we will need the following Lemma which is also used extensively in later sections.

\begin{lem}\label{dualcolontraceulrich}

Suppose $R$ is a Cohen-Macaulay ring of dimension $1$ with minimal multiplicity. Let $M$ be a torsion-free $R$-module and $I$ an ideal of $R$ containing a nonzerodivisor $x$. Then we have the following:
\begin{enumerate}
\item[$(1)$] $\Omega^1_R(M)$ is Ulrich.
\item[$(2)$] If $M$ has no nonzero free summand, then $M^*$ is Ulrich.
\item[$(3)$] If $I$ is not principal, then $(x:I)$ is Ulrich.
\item[$(4)$] If the ideal $J$ is Ulrich as an $R$-module then $JM$ is Ulrich.
\item[$(5)$] If $I$ is not principal, then $\tr_R(I)$ is Ulrich.
\item[$(6)$] If $M$ has no nonzero free summand, then $\tr_R(M)$ is Ulrich.

\end{enumerate}
\end{lem}
\begin{proof}
Extending the residue field if needed, we may suppose $R$ has an infinite residue field. For $(1)$, it follows from \cite[Corollary 1.2.5]{Av10} that $\Omega^1_R(M)$ has no nonzero free summands, and then \cite[Theorem B (1)]{KT19} implies that it is Ulrich. For $(2)$, as $M$ has no nonzero free summand, Corollary \ref{dualsummandweak} forces $M^*$ to have no nonzero free summand. Then $M^* \cong \Omega^2_R(\Tr_R(M))$ is a (minimal) syzygy of a torsion-free module, and so is Ulrich by $(1)$. For $(3)$, it is well-known that the map $I^* \to (x:I)$ given by $f \mapsto f(x)$ is an isomorphism \cite[Lemmas 2.4.1, 2.4.3]{HS06}, so $(3)$ follows from $(2)$. For $(4)$, let $y$ be a minimal reduction of $\m$. Then as $J$ is Ulrich, we have $yJ=\m J$, so $yJM=\m JM$, meaning $JM$ is Ulrich. For $(5)$, we have $x \tr_R(I)=(x:I)I$ from Proposition \ref{traceprop} (2), so in particular, $\tr_R(I) \cong (x:I)I$, and the claim follows from combining $(3)$ and $(4)$. Item $(6)$ follows from $(5)$ once we note $\tr_R(\tr_R(M))=\tr_R(M)$, and that $\tr_R(M)$ is principal if and only if $M$ has a nonzero free summand, see Proposition \ref{basictraceprop} (2).
\end{proof}

With these results in hand, we show that, if $R$ is Cohen-Macaulay of dimension $1$ with canonical module $\w$, and $R$ has minimal multiplicity, then we can attach an Ulrich module to any $R$-module $M$ via the following procedure:

\begin{theorem}\label{timesIUlrich}

Let $R$ be a Cohen-Macaulay local ring of dimension $1$ with minimal multiplicity and with canonical module $w_R$. Let $M$ be a finitely generated $R$-module. If $M$ has no nonzero free summand, then $(M \otimes_R w_R)/\Gamma_{\m}(M \otimes_R w_R)$ is an Ulrich module. If additionally, $M$ is torsion-free and $R$ has a canonical ideal $I$ (i.e. $R$ is generically Gorenstein), then $IM$ is Ulrich.

    
\end{theorem}

\begin{proof}
If $M$ is torsion, then $M \otimes_R \w$ is also torsion, so $(M \otimes_R \w)/\Gamma_{\m}(M \otimes_R \w)=0$ and there is nothing to prove. Otherwise, $\dim(M \otimes_R \w)=\dim(R)=1$ since $\w$ has full support, so $(M \otimes_R \w)/\Gamma_{\m}(M \otimes_R \w)$ is maximal Cohen-Macaulay. Moreover, applying $\Hom_R(-,\w)$ to the exact sequence 
\[0 \to \Gamma_{\m}(M \otimes_R \w) \to M \otimes_R \w \to (M \otimes_R \w)/\Gamma_{\m}(M \otimes_R \w) \to 0,\] we obtain, since $\Gamma_{\m}(M \otimes_R \w)$ is torsion, an isomorphism 
\[\Hom_R((M \otimes_R \w)/\Gamma_{\m}(M \otimes_R \w),\w) \cong \Hom_R(M \otimes_R \w,\w).\] But $\Hom_R(M \otimes_R \w,\w) \cong M^*$ by Hom-tensor adjointness and as $\Hom_R(\w,\w) \cong R$. As $(M \otimes_R \w)/\Gamma_{\m}(M \otimes_R \w)$ is torsion-free, it follows that 
\[(M \otimes_R \w)/\Gamma_{\m}(M \otimes_R \w) \cong \Hom_R(M^*,\w).\]

It follows from Lemma \ref{dualcolontraceulrich} (2) that $M^*$ is Ulrich, but then \cite[Theorem A (2)]{KT19} implies that $\Hom_R(M^*,\w)$, and thus $(M \otimes_R I)/\Gamma_{\m}(M \otimes_R \w)$ is Ulrich.

If $I$ is a canonical ideal and $M$ is torsion-free, then $IM$ is also torsion-free. Since $I$ is $\m$-primary, we have that $\Tor^R_1(M,R/I)$ is torsion. It thus follows from the exact sequence
\[0 \to \Tor^R_1(M,R/I) \to M \otimes_R I \to IM \to 0\]
that $IM \cong (M \otimes_R I)/\Gamma_{\m}(M \otimes_R I)$, completing the proof.
\end{proof}








\begin{prop}\label{reductionext}
Suppose $R$ is a Cohen-Macaulay ring of dimension $1$ having minimal multiplicity, let $M$ and $N$ be maximal Cohen-Macaulay $R$-modules. Set $L:=\Ann_R(\Ext^1_R(M,N))$. The following statements hold.
\begin{enumerate}
\item We have $\bar{L}=\m$ if and only if $L=\m$. 

\item If $y$ is a minimal reduction of $\m$, $N=R$, and $M/yM$ has no nonzero free $R/y$-summands, then $\dim_k(\Ext^1_R(M,R))=r(R)\mu_R(M)-e_R(M^*)$ whenever $\bar{L}=\m$.
\end{enumerate}

    
\end{prop}

\begin{proof}

Extending the residue field if needed, we may suppose $\m$ has a minimal reduction $y$, and it suffices to show $y\Ext^1_R(M,N)=0$ if and only if $\m \Ext^1_R(M,N)=0$. It's clear that $y$ annihilates $\Ext^1_R(M,N)$ if $\m$ does, so we need only concern ourselves with the reverse implication. For this, applying $\Hom_R(-,N)$ to the natural exact sequence
\[0 \rightarrow \Omega^1_R(M) \rightarrow R^{\oplus \mu_R(M)} \xrightarrow{p} M \rightarrow 0,\]
we obtain an exact sequence of the form
\[0 \to C \to \Hom_R(\Omega^1_R(M),N) \to \Ext^1_R(M,N) \to 0\]
where $C:=\coker(\Hom(p,N))$. It follows from Lemma \ref{dualcolontraceulrich} (1) that $\Omega^1_R(M)$ is Ulrich, and then so is $\Hom_R(\Omega^1_R(M),N)$ by \cite[Theorem A (2)]{KT19}. But then $y \Ext^1_R(M,N)=0$ if and only if $y\Hom_R(\Omega^1_R(M),N) \subseteq C$, and as $\Hom_R(\Omega^1_R(M),N)$ is Ulrich, $y\Hom_R(\Omega^1_R(M),N)=\m \Hom_R(\Omega^1_R(M),N)$. So $\m \Hom_R(\Omega^1_R(M),N) \subseteq C$, and it follows $\m \Ext^1_R(M,N)=0$, as desired. 

Now, suppose $N=R$ and $M/yM$ has no nonzero free $R/y$ summand. With $y\Ext^1_R(M,R)=0$, we apply $\Hom_R(-,R)$ to the exact sequence
\[0 \rightarrow M \xrightarrow{\cdot y} M \rightarrow M/yM \rightarrow 0,\]
and obtain an exact sequence
\[0 \rightarrow M^* \xrightarrow{\cdot y} M^* \rightarrow \Ext^1_R(M/yM,R) \rightarrow \Ext^1_R(M,R) \xrightarrow{\cdot y} \Ext^1_R(M,R).\]
But as $y \Ext^1_R(M,R)=0$, the sequence
\[0 \rightarrow M^* \xrightarrow{\cdot y} M^* \rightarrow \Ext^1_R(M/yM,R) \rightarrow \Ext^1_R(M,R) \rightarrow 0\]
is exact. By \cite[Lemma 1.2.4]{BH93}, we have $\Ext^1_R(M/yM,R) \cong \Hom_{R/y}(M/yM,R/y)$. It follows from Lemma \ref{dualcolontraceulrich} (2) that $\Hom_{R/y}(M/yM,R/y)$ is an Ulrich $R/y$-module, that is that $\m \Hom_{R/y}(M/yM,R/y)=0$. Now by using Hom-tensor adjointness and \cite[Proposition 3.2.12]{BH93}, we get that $\Hom_{R/y}(M/yM, R/y)$ is Matlis dual to $M/yM\otimes_{R/y}E_{R/y}(k)$, and hence \[l_R(\Hom_{R/y}(M/yM,R/y))=l_R(M/yM\otimes_{R/y}E_{R/y}(k))=\mu_R(M/yM\otimes_{R/y}E_{R/y}(k))\]
\[=\mu_R(M/yM)\mu_R(E_{R/y}(k))=r(R)\mu_R(M).\]
As $l_R(M^*/yM^*)=e_R(M^*)$, the remaining claim follows from additivity of length.    
\end{proof}

\begin{cor}\label{idealcase}
Suppose $R$ is a Cohen-Macaulay ring of dimension $1$ with minimal multiplicity. Suppose $J$ is an $\m$-primary ideal of $R$ that is not principal, and let $L:=\Ann_R(\Ext^1_R(J,R))$. Then $\bar{L}=\m$ if and only if $L=\m$. Moreover, when either of these equivalent conditions hold, we have $\dim_k(\Ext^1_R(J,R))=r\mu_R(J)-r-1$.




\end{cor}

\begin{proof}

As $I$ is an $\m$-primary ideal, we have $l_R(I/yI)=l_R(R/y)=e(R)$, so if $I/yI$ has a free $R/y$ summand, then $I/yI \cong R/y$. But this would force $I$ to be principal, which we suppose is not the case. The claim then follows from Proposition \ref{reductionext}, noting that, $e_R(I^*)=e(R)=r+1$ as $I$ has rank $1$. \end{proof}

\begin{lem}\label{torsiondual}
Suppose $R$ is a Cohen-Macaulay local ring of dimension $1$ with canonical module $\w$. If $I$ is an $\m$-primary ideal, then there is an isomorphism $\Tor^R_1(I,R/I) \cong \Hom_R(\Ext^1_R(I,I^{\vee}),E_R(k))$. In particular, $\Tor^R_1(I,R/I) \cong \Ext^1_R(I,I^{\vee})$ if $\m \Ext^1_R(I,I^{\vee})=0$.  
\end{lem}

\begin{proof}
 By dimension shifting, we have $\Ext^1_R(I,I^{\vee}) \cong \Ext^2_R(R/I,I^{\vee})$. As $I$ is $\m$-primary, $R/I$ has finite length, and then the claim follows from \cite[Lemma 3.5 (2)]{KO22}. 
\end{proof}

\begin{cor}\label{extvec}
Suppose $R$ is a Cohen-Macaulay ring of dimension $1$ with minimal multiplicity, and suppose $R$ admits a canonical ideal $I$. Let $L:=\Ann_R(\Ext^1_R(I,R))$. If $\bar{L}=\m$, then $\Ext^1_R(I,R) \cong k^{\oplus (r^2-r-1)}$ and $\delta_1(I) \cong k^{\oplus({r \choose 2}-1)}$. Moreover, in this situation, the natural exact sequences 
\[0 \rightarrow \Tor^R_1(I,R/I) \rightarrow I \otimes_R I \rightarrow I^2 \rightarrow 0\]
and
\[0 \rightarrow \delta_1(I) \rightarrow S^2_R(I) \rightarrow I^2 \rightarrow 0\]
are split.
    
\end{cor}

\begin{proof}

It follows from Corollary \ref{idealcase} that $\Ext^1_R(I,R) \cong k^{\oplus (r^2-r-1)}$. From Lemma \ref{torsiondual}, we have $\Tor^R_1(I,R/I) \cong k^{\oplus (r^2-r-1)}$ as well, and then Lemma \ref{thetorthing} (3) forces $\m \delta_1(I)=0$ with $\dim_k(\delta_1(I))=r^2-r-1-\displaystyle {r \choose 2}={r \choose 2}-1$. 

By Theorem \ref{timesIUlrich}, we have that $I^2$ is Ulrich, so $\mu_R(I^2)=e(R)=r+1$. For the exact sequence
\[0 \to \Tor^R_1(I,R/I) \to I \otimes_R I \to I^2 \to 0\]
we observe that $\mu_R(I \otimes_R I)=r^2=r+1+r^2-r-1=\mu_R(\Tor^R_1(I,R/I))+\mu_R(I^2)$. It follows that a minimal generating set for $\Tor^R_1(I,R/I)$ is part of a minimal generating set for $I \otimes_R I$, but these are minimal generators of $I\otimes_R I$ which belong to its socle, so the sequence splits. That the sequence
\[0 \to \delta_1(I) \to S^2_R(I) \to I^2 \to 0\]
splits follows from a similar argument.
\end{proof}

The next proposition, while simple, gives some criteria under which an ideal that is Ulrich as an $R$-module must be equal to $\m$, and will be used repeatedly.



\begin{prop}\label{idealulrichred}

Let $R$ be a Cohen-Macaulay local ring of dimension $1$. Let $I$ be an $\m$-primary ideal of $R$ that is Ulrich as an $R$-module. If $\bar{I}=\m$, then $I=\m$.     
\end{prop}

\begin{proof}
Extending the residue field if needed, we may suppose $k$ is infinite. Then as $\bar{I}=\m$, we have that $I$ contains a minimal reduction $y$ of $\m$. Then there is an exact sequence 

\[0 \to (y)/yI  \to I/yI \to I/(y) \to 0.\]

Since $I$ is Ulrich as an $R$-module, $I/yI \cong k^{\oplus \mu_R(I)}$. In particular, $\m \subseteq \Ann_R((y)/yI)=(yI:y)$. But since $y$ is a nonzerodivisor, $(yI:y)=I$, so $I=\m$.    
\end{proof}

\subsection{Numerical Semigroup Rings}\label{numericalsemigroupsection}

By a numerical semigroup ring, we mean ambiguously either the positively graded ring $k[t^{a_1},\dots,t^{a_n}]$ or its completion $k[\![t^{a_1},\dots,t^{a_n}]\!]$ where $a_1<a_2<\ldots<a_n$. All the properties we are interested in will ascend to and descend from the completion in this case, so we will work in practice with the positively graded ring $k[t^{a_1},\dots,t^{a_n}]$, and will frequently exploit its graded structure. For numerical semigroup rings, it is well-known that $e(R)=a_1$ \cite[Proposition 2.10]{rosales2009numerical} and that $\mathfrak{c}_R=(t^{c+i})_{0\leq i\leq e(R)-1}$, where $c-1$ is the largest integer not present in the semigroup generated by the integers $a_i$ (see for instance \cite[Remark 2.1]{MM23}). Further, the type of the numerical semigroup ring is given by the cardinality of the set of pseudo-Frobenius numbers (see \cite[Theorem 3.1]{S18} or \cite[Theorem 3.2]{MM23} for further details). A key point to the structure of $k[t^{a_1},\dots,t^{a_n}]$ is that a homogeneous element of degree $d$ is uniquely determined up to multiplication by an element of $k$ as $t^d$. The fact leads to several deeper consequences, beginning with the following elementary lemma that serves as a key point in what follows.

\begin{lem}\label{decompcolonideal}

Suppose $R$ is a numerical semigroup ring and that $L$ is a homogeneous ideal of $R$ with minimal homogeneous generating set $(x_1,\dots,x_m)$. For any $i$ and any homogeneous $a \in ((x_1,\dots,\hat{x_i},\dots,x_m):x_i)$, we have that $a \in (x_j:x_i)$ for some $i \ne j$. In particular,
\[((x_1,\dots,\hat{x_i},\dots,x_m):x_i)=\sum_{j \ne i} (x_j:x_i).\]

\end{lem}

\begin{proof}

If $a$ is a homogeneous element of $(x_1,\dots,\hat{x_i},\dots,x_m):x_i)$, then $ax_i$ is a homogeneous element of $(x_1,\dots,\hat{x_i},\dots,x_m)$. But every homogeneous element of $(x_1,\dots,\hat{x_i},\dots,x_m)$ is a multiple of $x_j$ for some $j \ne i$, so $a \in (x_j:x_i)$. The latter claim follows as we evidently have $\sum_{j \ne i} (x_j:x_i) \subseteq ((x_1,\dots,\hat{x_i},\dots,x_m):x_i)$.

\end{proof}

\begin{prop}\label{colontocycle}

Let $I$ be an ideal of $R$ of positive grade and suppose $(x_1,\dots,x_m)$ is a minimal generating set for $I$ where each $x_i$ is a nonzerodivisor (such a generating set always exists by prime avoidance). For each $i \ne j$, there is an injective map $p_{ij}:(x_i:x_j) \to Z_1(I)$ given by $a \mapsto ae_j-\dfrac{ax_j}{x_i}e_i$. Moreover, we have the following:

\begin{enumerate}

\item[$(1)$] The map 
\[q_i
:\bigoplus_{j \ne i} (x_j:x_i) \to Z_1(I)\] and the map 
\[q^j
:\bigoplus_{i \ne j} (x_j:x_i) \to Z_1(I)\] 
given on components by $p_{
ji}$ are injective for a fixed $i$ or a fixed $j$, respectively.

\item[$(2)$] Each map $p_{ji}$ restricts to a map $(x_jI:x_i) \cap I \to Z_1(I) \cap IR^{\oplus m}$.

\end{enumerate}

\end{prop}

\begin{proof}

The existence and injectivity of the maps $p_{ij}$ is evident from the fact that each $x_i$ is a nonzerodivisor. For claim $(1)$, if $\sum_{j \ne i} p_{ji}(z_j)=\sum_{j \ne i} z_je_i-\sum_{j \ne i}\dfrac{ax_j}{x_i}e_i=0$, then we immediately obtain that $z_j=0$. 
That $q^j$ is injective follows similarly.

For claim $(2)$, if $a \in (x_jI:x_i) \cap I$, then both $a$ and $\dfrac{ax_j}{x_i}$ are in $I$, so $p_{ij}(a) \in Z_1(I) \cap IR^{\oplus m}$, as desired.
    \end{proof}

Lemma \ref{decompcolonideal} and Proposition \ref{colontocycle} put severe constraints on what the presentation matrix of a homogeneous ideal in a numerical semigroup ring can look like:

\begin{prop}\label{matrixtwocolsup}
Suppose $R$ is a numerical semigroup ring and that $L$ is a homogeneous ideal in $R$ with minimal homogeneous generating set $(x_1,\dots,x_m)$. For each $i \ne j$, the map $p_{ji}:(x_j:x_i) \to Z_1(I)$ described in Proposition \ref{colontocycle} is homogeneous of degree $|x_i|$. Moreover, we have the following:

\begin{enumerate}
\item[$(1)$] The maps 
\[q_i:\bigoplus_{j \ne i} (x_j:x_i)(-|x_i|) \to Z_1(I)\]
and 
\[q^j:\bigoplus_{i \ne j} (x_j:x_i)(-|x_i|) \to Z_1(I)\]
are homogeneous of degree $0$. 
\item[$(2)$] The map $q:\bigoplus_{i=1}^m \bigoplus_{j \ne i} (x_j:x_i)(-|x_i|) \to Z_1(I)$ given on components as $p_{ji}(-|x_i|)$ is surjective.
    
\end{enumerate}

In particular, $L$ has a minimal presentation matrix whose columns each have exactly two nonzero entries.

\end{prop}

\begin{proof}

If $a \in (x_j:x_i)$ is a homogeneous element, then as the degree of $e_i$ is $|x_i|$ for each $i$, we have $|ae_i|=|a|+|x_i|$ and similarly that $|\dfrac{ax_i}{x_j}e_j|=|a|+|x_i|-|x_j|+|x_j|=|a|+|x_i|$. It follows that $p_{ji}(a)=ae_i-\dfrac{ax_i}{x_j}e_j$ is a homogeneous element of $Z_1(I)$ of degree $|a|+|x_i|$, so $p_{ji}$ is a graded map of degree $|x_i|$. Claim $(1)$ now follows immediately. 

For claim $(2)$, we proceed by induction on $m$, noting the claim is clear when $m=1$. Suppose $z:=\sum_{i=1}^m a_ie_i$ is a homogeneous element in $Z_1(I)$ so in particular each $a_i$ is homogeneous. Then as $\sum_{i=1}^m a_ix_i=0$, we have $a_1 \in ((x_2,\dots,x_m):x_1)$, and Lemma \ref{decompcolonideal} implies that $a_1 \in (x_{j_1}:x_1)$ for some $j_1 \ne 1$. Then $a_1x_1=\dfrac{a_1x_1}{x_{j_1}}x_{j_1}$, we have $\sum_{i=1}^m a_ix_i=\sum_{i=2}^m a_ix_i+\dfrac{a_1x_1}{x_{j_1}}x_{j_1}=0$ so setting $\bar{I}=(x_2,\dots,x_m)$, we have $\sum_{i=2}^m a_ie_i+\dfrac{a_1x_1}{x_{j_1}}e_{j_1} \in Z_1(\bar{I})$. By induction $\sum_{i=2}^m a_ie_i+\dfrac{a_1x_1}{x_{j_1}}e_{j_1}$ is in the image of the map 
\[\bar{q}:\bigoplus^m_{i=2} \bigoplus_{j \ne i} (x_j:x_i)(-|x_i|) \to Z_1(\bar{I})\]
given on components as $p_{ji}$. In particular, $\sum_{i=2}^m a_ie_i+\dfrac{a_1x_1}{x_{j_1}}e_{j_1}$ viewed as an element of $Z_1(I)$ is in the image of $q$. But $a_1e_1-\dfrac{a_1x_1}{x_{j_1}}e_{j_1}$ is in the image of $p_{j_1 1}$, in particular is in the image of $q$. It follows that $z=a_1e_1-\dfrac{a_1x_1}{x_{j_1}}e_{j_1}+\sum_{i=2}^m a_ie_i+\dfrac{a_1x_1}{x_{j_1}}e_{j_1} \in \im(q)$. Therefore, $q$ is surjective. Note $\bigoplus_{i=1}^m \bigoplus_{j \ne i} (x_j:x_i)(-|x_i|)$ has a minimal homogeneous generating set inherited from the minimal homogeneous generators of the $(x_j:x_i)$, which then maps to a homogeneous generating set for $Z_1(I)$ and may then be trimmed to a minimal one. As the image under $q$ of any such generator is a column vector in $R^{\oplus m}$ with two elements of support, it follows that $L$ has a presentation matrix whose columns each have two elements of support.
\end{proof}



\section{Main Results}\label{main}

In this section we explore when trace ideals and colon ideals contain $\m$ for $R$ a Cohen-Macaulay ring of minimal multiplicity. Our focus lies on the case when $R$ has small dimension. We first note two preliminary cases of interest:
\begin{prop}\label{artinian}
Suppose $R$ is Artinian and has minimal multiplicity, i.e., that $\m^2=0$. Then $R$ is nearly Gorenstein.
\end{prop}

\begin{proof}
Since $\m^2=0$, $\m \cong k^{\oplus \mu_R(\m)}$. Then there is an $R$-module isomorphism $\psi:\Hom_R(\m,\w) \cong \m$. Consider the natural short exact sequence
\[0 \rightarrow \m \xrightarrow{i} R \rightarrow k \rightarrow 0.\]
Applying $\Hom_R(-,\w)$ to this exact sequence, we obtain a short exact sequence of the form
\[0 \rightarrow k \rightarrow \w \xrightarrow{\Hom(i,\w)} \Hom_R(\m,w_R) \rightarrow 0\]

Then the composition $q:=i \circ \psi \circ \Hom_R(i,\w)$ is a map $\w \to R$ whose image is $\m$. So for any $y \in \m$, there is an $x \in \w$ so that $q(x)=y$, which means $\theta_R^{\omega_R}(q \otimes x)=y$ (see \Cref{tracedef}). It follows that $\m \subseteq \tr_R(\w)$, so that $R$ is nearly Gorenstein. 
    \end{proof}

\begin{prop}\label{standardgradedreduced}
Suppose $R$ is a standard graded reduced $k$-algebra of dimension $1$ and that $k$ is algebraically closed. If $R$ has minimal multiplicity, then $\mathfrak{c}_R=\m$. In particular, $R$ is nearly Gorenstein.

\end{prop}

\begin{proof}
Let $e:=e(R)$. Then by Nullstellensatz, $R$ is the coordinate ring of $e$ points $P_1,\dots,P_e$ in $\mathbb{P}^{e-1}_k$, so we may write $R \cong k[x_1,\dots,x_e]/I$ where $I=\bigcap_{i=1}^e I(P_i)$ contains no linear form.  

By \cite[Corollary 2.1.13]{HS06}, the conductor is given by 
\[\mathfrak{c}_R=\sum_{j=1}^e I(P_1,\dots,\hat{P}_j,\dots,P_e).\]
Note that any $e-1$ points in $\mathbb{P}^{e-1}_k$ lie on a common hyperplane, so for any $j$, $I(P_1,\dots,\hat{P}_j,\dots,P_e)=\bigcap_{i \ne j} I(P_i)$ contains a homogeneous linear form $f_j$. As $I$ contains no linear forms, it follows that $f_j \notin I(P_j)$ for each $j$. Set $f:=f_1+\cdots+f_e$. If $f \in P_j$ for any $j$, then as $f_i \in P_j$ for all $i \ne j$, this would force $f_j \in P_j$, which is not the case. It follows that $f \notin \bigcup_{i=1}^e I(P_i)$, and so $f$ is a homogeneous linear form in $\mathfrak{c}_R$ that is a nonzerodivisor on $R$. We claim $f$ is a reduction of $\m$. Indeed, from the exact sequence
\[0 \rightarrow R(-1) \xrightarrow{ \cdot f} R \rightarrow R/fR \rightarrow 0,\]
we have that $\epsilon_{R/fR}(t)=H_{R/fR}(t)=(1-t)H_R(t)=\epsilon_R(t)$. So $e(R)=\epsilon_R(1)=\epsilon_{R/fR}(1)=l_R(R/fR)$, and it follows from e.g. \cite[Lemma 2.2]{De16} that $f$ is a reduction of $\m$, as claimed. But it well-known that $\mathfrak{c}_R$ is Ulrich as an $R$-module (see e.g. \cite[Corollary 4.10]{DS23}), and then Proposition \ref{idealulrichred} forces $\mathfrak{c}_R=\m$, as desired. For the claim about nearly Gorensteinness, we note that $\tr_R(\w)$ contains $\mathfrak{c}_R=\m$ for example by \cite[Corollary 3.6]{DS23}, so $R$ is nearly Gorenstein. 
\end{proof}


In light of Propositions \ref{artinian} and \ref{standardgradedreduced}, one may be tempted to believe that Cohen-Macaulay rings of minimal multiplicity will always be nearly Gorenstein. This is far from the case, as seen in the following examples:
\begin{example}[{\cite[Corollary 4.4]{HT19}}]\label{minmultadjoingraded}
Suppose $R$ is positively graded. Then the following are equivalent:
\begin{enumerate}
    \item[$(1)$] $R[x_1,\dots,x_n]$ is nearly Gorenstein.
    \item[$(2)$] $R$ is Gorenstein.
\end{enumerate}
\end{example}

\begin{example}[{\cite[Proposition 4.5]{HT19}}]\label{minmultadjoinlocal}
Suppose $R$ is a local ring. Then the following are equivalent:
\begin{enumerate}
    \item[$(1)$] $R[\![x_1,\dots,x_n]\!]$ is nearly Gorenstein.
    \item[$(2)$] $R$ is Gorenstein.
\end{enumerate}
\end{example}

Examples \ref{minmultadjoingraded} and \ref{minmultadjoinlocal} imply that if $R$ is Cohen-Macaulay but not Gorenstein and has minimal multiplicity, then $R[t]$ in the graded setting and $R[\![t]\!]$ in the local setting are Cohen-Macaulay with minimal multiplicity but are not nearly Gorenstein.

We give some other concrete examples in the setting of numerical semigroup rings:

\begin{example}



Let $R=k[\![t^e,t^{\ell e+1},t^{\ell e+2}, \ldots, t^{\ell e+e-1}]\!]$ for any $e \ge 3$ and $\ell\geq 2$. Note that $R$ has minimal multiplicity. 
It follows from \cite{rosales2009numerical} (cf. \cite[Proof of Corollary 6.2]{He23}) that the largest integer $i$ for which $t^i \notin R$ is $\ell e-1$, and so $\mathfrak{c}_R=(t^{\ell e},t^{\ell e+1},\dots,t^{\ell e+e-1}) \ne \m$. It follows from \cite[Corollary 6.2]{He23} that $R$ is far-flung Gorenstein, i.e., that $\tr_R(\w)=\mathfrak{c}_R$, so $R$ is not nearly Gorenstein.
    
\end{example}

Much more can be said for numerical semigroup rings than in the general case. In large part, this is owed to the graded structure and in particular to the results proven in Section \ref{numericalsemigroupsection}. To this end, we will let $R$ denote either the positively graded numerical semigroup ring $k[t^{a_1},\dots,t^{a_n}]$ or the complete numerical semigroup ring $k[\![t^{a_1},\dots,t^{a_n}]\!]$ for the remainder of this section. As the properties we are interested ascend to and descend from the completion, we will work in practice with $R=k[t^{a_1},\dots,t^{a_n}]$. We set $y:=t^{a_1}$.

The following criterion serves as a key point in the proof of the main theorem of this section.

\begin{theorem}\label{setup}
Let $I$ be a homogeneous ideal for $R$ and let $(x_1,\dots,x_r)$ be a minimal homogeneous generating set for $I$ with $|x_i|<|x_{i+1}|$ for each $1 \le i \le r-1$. Suppose $(x_{\ell}:x_j)(x_j:x_1) \subseteq (y)$ for all $\ell$ and $j$, e.g. $(x_j:x_1)$ is Ulrich for all $j$. If $yB_1(I) \subseteq IZ_1(I)$, then $y \in (x_1:I)$.

\end{theorem}

\begin{proof}

For any $j \ne 1$, we have $x_1e_j-x_je_1 \in B_1(I)$. Then by assumption, we have $yx_1e_j-yx_je_1 \in IZ_1(I)$. As $yx_1e_j-yx_je_1$ is homogeneous, there exist homogeneous $b^j_1,\dots,b^j_m \in I$ and $z^j_1,\dots,z^j_m \in Z_1(I)$ with 
\[yx_1e_j-yx_je_1=\sum_{i=1}^m b^j_iz^j_i.\]
As $yx_1$ is a homogeneous element of smallest degree in $\m I$, it must be that $b^j_i=u_jx_1$ for some $u_j \in k$ and that $j$th entry of $z^j_i$ is $w_jy$ for some $w_j \in k$. Since $\sum_{k=1}^m(z^{j}_i)_{k}x_k=0$, we know that $w_jy=(z^j_i)_j\in ((x_1,\ldots,\hat{x}_j,\ldots,x_m):x_j)$. Hence Lemma \ref{decompcolonideal} shows that for each $j>1$, there is an $\ell_j$ with $y\in (x_{\ell_j}:x_j)$ and by degree considerations, we get that $\ell_j <j$.

We proceed by induction on $j$ to show that $y \in (x_1:x_j)$ for all $2 \le j \le n$. That $y \in (x_1:x_2)$ follows at once from above, so the base case is established. Now suppose $y \in (x_1:x_s)$ for all $2 \le s \le j-1$. By above, we have that $y \in (x_{\ell}:x_j)$ for some $1 \le \ell<j$, and we also have from inductive hypothesis that $y \in (x_1:x_{\ell})$. If $\ell=1$ there is nothing to prove, so we may suppose $\ell>1$. Then $yx_j=sx_{\ell}$ and $yx_{\ell}=wx_1$ for some homogeneous $s,w \in R$. Note that $s,w \in \m$, as $|x_1|<|x_{\ell}|<|x_j|$. But then since $(x_{\ell}:x_j)(x_j:x_1) \subseteq (y)$ for all $\ell,j$, we have 
\[y^2x_j=ysx_{\ell}=swx_1=yvx_1\] where $v$ must be in $\m$ by degree considerations. But as $R$ is a domain, we have $yx_j=vx_1$. So we in fact have $y \in (x_1:x_j)$. Induction forces $y \in (x_1:x_j)$ for all $2 \le j \le n$, but then $y \in \bigcap_{j=2}^n (x_1:x_j)=(x_1:I)$, as desired.
\end{proof}

\begin{theorem}\label{keypoint}

Suppose $R$ has minimal multiplicity and that $I$ is a homogeneous nonprincipal ideal in $R$. Let $x_1,\dots,x_n$ be a minimal homogeneous generating set for $I$ with $|x_1|<|x_2|<\cdots <|x_n|$. Let $B_1(I)$ and $Z_1(I)$ denote the Koszul boundaries and Koszul cycles with respect to this generating set. Then $yB_1(I) \subseteq IZ_1(I)$ if and only if $y \in (x_1:I)$. 

\end{theorem}

\begin{proof}
First suppose that $y \in (x_1:I)$. Then by Proposition \ref{traceprop} (3) we have $y \in \tr_R(I)$. Then combining \cite[Lemma 2.2]{DK21} and Proposition \ref{traceprop} (4), we have that $y \Ext^1_R(I,I^{\vee})=0$, and from Lemma \ref{torsiondual} we have $y \Tor^R_1(I,R/I)=0$. From Proposition \ref{thetorthing}, it follows that that $y (B_1(I)/IZ_1(I))=0$. In particular, we have $yB_1(I) \subseteq IZ_1(I)$.

Now suppose instead that $yB_1(I) \subseteq IZ_1(I)$. Note that $(x_j:x_1)=(x_j:(x_1,x_j))$, and then Lemma \ref{dualcolontraceulrich} (3) forces $(x_j:x_1)$ to be Ulrich. Theorem \ref{setup} then forces $y \in (x_1:I)$, as desired. 

\end{proof}


\begin{theorem}\label{generalideals}

Suppose $R$ has minimal multiplicity and that $I$ is a homogeneous nonprincipal ideal in $R$. Let $x_1,\dots,x_n$ be a minimal homogeneous generating set for $I$ with $|x_1|<|x_2|<\cdots <|x_n|$. Then the following are equivalent:

\begin{enumerate}

\item[$(1)$] $(x_1:I)=\m$.
\item[$(2)$] $\tr_R(I)=\m$.
\item[$(3)$] $\m \Ext^i_R(I,I^{\vee})=0$ for all $i>0$.
\item[$(4)$] $\m \Ext^1_R(I,I^{\vee})=0$.
\item[$(5)$] $\m \Tor^R_1(I,R/I)=0$.
\item[$(6)$] $\m B_1(I) \subseteq IZ_1(I)$.
\item[$(7)$] $y B_1(I) \subseteq IZ_1(I)$.
\item[$(8)$] $y \in (x_1:I)$.
\item[$(9)$] $\m \bigwedge^2_R(I)=0$.
\item[$(10)$] $y\Ext^i_R(I,I^\vee)=0$ for all $i>0$.
\item[$(11)$] $y\Ext^1_R(I,I^\vee)=0$.
\item[$(12)$] $y \Tor^R_1(I,R/I)=0$.
\item[$(13)$] $y \bigwedge^2_R(I)=0$.
\item[$(14)$] $y \in \tr_R(I)$.
\item[$(15)$] There is an ideal $L$ with $I \cong L$ such that $y \in L$.
\item[$(16)$] There is an isomorphism $Z_1(I) \cong \bigoplus_{i=2}^n \m(-|x_i|)$ of graded $R$-modules.

\end{enumerate}

\end{theorem}

\begin{proof}

We note $(1) \Rightarrow (2)$ follows from Proposition \ref{traceprop} (3) and $(2) \Rightarrow (3)$ follows from combining \cite[Lemma 2.2]{DK21} and Proposition \ref{traceprop} (4).

The implications $(3) \Rightarrow (4)$ and $(6) \Rightarrow (7)$ are clear, $(4) \Rightarrow (5)$ and $(5) \Rightarrow (6)$ follow from Proposition \ref{torsiondual} and Proposition \ref{thetorthing} respectively, and $(7) \Rightarrow (8)$ is the content of Theorems \ref{setup} and \ref{keypoint}. We now show $(8) \Rightarrow (1)$. Indeed, it follows from Lemma \ref{dualcolontraceulrich} (3) that $(x_1:I)$ is Ulrich as an $R$-module, but since $(x_1:I)$ contains the reduction $y$, Proposition \ref{idealulrichred} (1) implies that $(x_1:I)=\m$. The equivalence of items $(1)-(8)$ is now established. We note that $(9) \Rightarrow (6)$ from Proposition \ref{thetorthing}, and we now show $(1) \Rightarrow (9)$. Indeed, if $(x_1:I)=\m$, then for any $a \in \m$, we have $\dfrac{ax_j}{x_1} \in R$ for all $j$, but in fact, as the $x_j$ are minimal generators of $I$, we must have $\dfrac{ax_j}{x_1} \in \m$, or else the equation $\dfrac{ax_j}{x_1}x_1=ax_j$ would force $x_1 \in (x_j)$. Then $\dfrac{ax_j}{x_1^2}x_i=\dfrac{ax_ix_j}{x_1^2} \in R$ as well, so for any $i<j$, we have $a(x_i \wedge x_j)=\dfrac{ax_ix_j}{x_1^2}(x_1 \wedge x_1)=0$. Therefore, $\m \bigwedge^2_R(I)=0$, and the equivalence of $(1)-(9)$ is established.

Now we note $(3)$ clearly implies $(10)$, and similar to above, $(10) \Rightarrow (11)$ is clear, while $(11) \Rightarrow (12)$ follows from Lemma \ref{torsiondual}. To see $(12) \Rightarrow (13)$, we note $(12) \Rightarrow (7)$ from Proposition \ref{thetorthing}, and clearly $(9) \Rightarrow (13)$. As we have already established $(7) \Rightarrow (9)$, $(12) \Rightarrow (13)$ follows. But $(13)$ also implies $(7)$ by Proposition \ref{thetorthing}, while $(2) \Rightarrow (14)$ is immediate. That $(14) \Rightarrow (2)$ follows combining Lemma \ref{dualcolontraceulrich} $(5)$ and Proposition \ref{idealulrichred}, so the equivalence of $(1)-(14)$ is established.

We next show $(8) \Rightarrow (15)$. Indeed, if $y \in (x_1:I)$, then we set $L:=yI/x_1 \subseteq R$. Then $L \cong I$, and as $x_1 \in I$, we have $yx_1/x_1=y \in L$. But now $(15) \Rightarrow (14)$ holds since the trace ideal is invariant on the isomorphism class of a module, and we thus have $L \subseteq \tr_R(L)=\tr_R(I)$, so $y \in \tr_R(I)$.

Finally, we prove $(1) \Rightarrow (16)$ and $(16) \Rightarrow (8)$. First suppose $(x_1:I)=\m$. Then $(x_1:x_j)=\m$ for all $j=2,\dots,n$. We claim the graded injection $q^i:\bigoplus_{j=2}^n (x_1:x_j)(-|x_j|) \to Z_1(I)$ of Proposition \ref{matrixtwocolsup} is surjective. Note from Lemma \ref{dualcolontraceulrich} (1) that $Z_1(I)$ is Ulrich, and as $(x_1:x_j)=(x_1:(x_1,x_j))$, Lemma \ref{dualcolontraceulrich} (3) forces each $(x_1:x_j)$ to be Ulrich. In particular, $\bigoplus_{j=2}^n (x_1:x_j)(-|x_j|)$ and $Z_1(I)$ are Ulrich modules of the same rank $n-1$, and so have the same minimal number of generators. It thus suffices to show that $q^i \otimes k$ is injective, as then $\coker(q^i) \otimes_R k=0$. If not, there is a minimal homogeneous generator $v=(v_2,\dots,v_n)$ of $\bigoplus_{j=2}^n (x_1:x_j)(-|x_j|)$ for which $q^i(v)=-\sum_{i=2}^n \dfrac{v_ix_i}{x_1}e_1+\sum_{i=2}^nv_ie_i \in \m Z_1(I)=yZ_1(I)$. So $q^i(v)=yz$ for some $z=\sum_{i=1}^n z_ie_i \in Z_1(I)$. Being a homogeneous generator, we have some $v_i=ct^{a_{\ell}}$ for some $\ell$ and some nonzero $c \in k$. But then $ct^{a_{\ell}}=yz_i$, and as $z_i$ must be in $\m$, this would force $t^{a_{\ell}} \in \m^2$, which is not the case. It follows that $q^i \otimes k$  is injective, from which we have that $q^i$ is surjective, establishing that $(1) \Rightarrow (16)$.

Finally, we show $(16) \Rightarrow (8)$. Now suppose there is a graded isomorphism $\theta:\bigoplus_{i=2}^n \m(-|x_i|) \to Z_1(I)$. Then for all $i=2,\dots,n$, $Z_1(I)$ contains a homogeneous element $z_i:=\sum_{j=1}^n r^i_je_j$ of degree $|y|+|x_i|$. In particular, if $j>i$, we must have $r^i_j=0$, since each $r^i_j \in \m$, so $|r^i_j|+|x_j| \ge |y|+|x_j|>|y|+|x_i|$ when $r^i_j$ is nonzero. Then for each $i$, there is an $\ell_i \le i$, with $|r^i_{\ell_i}|+|x_{\ell_i}|=|y|+|x_i|$. When $\ell_i=i$, we have that $r^i_i=cy$ for some $c \in k$, and as $z_i \in Z_1(I)$, it follows from Lemma \ref{decompcolonideal} that $y \in (x_{j_i}:x_i)$ for some $j_i<i$. On the other hand, when $\ell_i<i$, we have already that $y \in (x_{\ell_i}:x_i)$. So in either case we have that for all $i=2,\dots,n$, there is a $j_i<i$ for which $y \in (x_{j_i}:x_i)$. Following the inductive process outlined in the proof of Theorem \ref{setup}, this forces $y \in (x_1:I)$, so $(16) \Rightarrow (8)$, and the proof is complete.

\end{proof}

Theorem \ref{generalideals} has the following fundamental consequence that to the best of our knowledge is new:

\begin{cor}\label{thisiscool}
Suppose $R$ is a numerical semigroup ring that is not a DVR. Then $R$ has minimal multiplicity if and only if $Z_1(\m) \cong \m^{\oplus e(R)-1}$. 
\end{cor}

\begin{proof}
First suppose $R$ has minimal multiplicity so that $e(R)=\mu_R(\m)$. As $R$ is not a DVR, $\tr_R(\m)=\m$, so Theorem \ref{generalideals} implies the existence of the desired isomorphism, and in fact, the isomorphism can be made graded with suitable shifts in the direct sum decomposition.

Now suppose instead that $Z_1(\m) \cong \m^{\oplus e(R)-1}$. By rank considerations we have $\rank(Z_1(\m))=\mu_R(\m)-1=e(R)-1$, so $R$ has minimal multiplicity.
\end{proof}

Applying Theorem 3.8 when $I$ is a canonical ideal for $R$, we immediately obtain the following, which gives an affirmative answer to Question \ref{daoetal} for numerical semigroup rings of minimal multiplicity:

\begin{theorem}\label{nearlygormainthm}

Suppose $R$ is a numerical semigroup ring of minimal multiplicity with canonical ideal $I$. Let $x_1,\dots,x_r$ be a minimal homogeneous generating set for $I$ with $|x_1|<|x_2|<\cdots <|x_r|$ and set $y=t^{a_1}$. Then the following are equivalent:

\begin{enumerate}

\item[$(1)$] $R$ is almost Gorenstein, i.e., $\m\subseteq (x_1:I)$.

\item[$(2)$] $R$ is nearly Gorenstein.

\item[$(3)$] $\m \Ext^i_R(I,R)=0$ for all $i>0$.

\item[$(4)$] $\m \Ext^1_R(I,R)=0$.

\item[$(5)$] $\m \Tor^R_1(I,R/I)=0$.

\item[$(6)$] $\m \bigwedge^2_R(I)=0$.

\item[$(7)$] $\m B_1(I) \subseteq IZ_1(I)$.

\item[$(8)$] $y B_1(I) \subseteq IZ_1(I)$.

\item[$(9)$] $y \in (x_1:I)$.

\item[$(10)$] $y\Ext^i_R(I,R)=0$ for all $i>0$.

\item[$(11)$] $y\Ext^1_R(I,R)=0$.

\item[$(12)$] $y \Tor^R_1(I,R/I)=0$.

\item[$(13)$] $y \bigwedge^2_R(I)=0$.

\item[$(14)$] $y \in \tr_R(I)$.

\item[$(15)$] There is an ideal $L$ with $I \cong L$ such that $y \in L$.

\item[$(16)$] There is an isomorphism $Z_1(I) \cong \bigoplus_{i=2}^r \m(-|x_i|)$ of graded $R$-modules. 

\end{enumerate}
    
\end{theorem}

Theorem \ref{nearlygormainthm} indicates that, for a numerical semigroup ring $R$ of minimal multiplicity, the nearly Gorenstein property is intimately connected to the annihilator of $\bigwedge^2_R(I)$ where $I$ is the canonical ideal of $R$, and in particular, it is sufficient that $t^{a_1} \bigwedge^2_R(I)=0$. In view of Proposition \ref{thetorthing}, it is then natural to ask whether $t^{a_1} \delta_1(I)=0$, or even $\m \delta_1(I)=0$, is sufficient to guarantee $R$ is nearly Gorenstein when $R$ has minimal multiplicity. The following example shows this is not the case:  

\begin{example}\label{mkillssyztype3}
Let $R=\mathbb{F}_2[\![t^4,t^9,t^{14},t^{15}]\!]$ with $\m=(t^4,t^9,t^{14},t^{15})$. Then $R$ has the canonical ideal $I:=(t^8,t^9,t^{14})$ with $\m \delta_1(I)=0$, but $\tr_R(I)=(t^8,t^9,t^{14},t^{15}) \ne \m$. So $R$ is not nearly Gorenstein.

\end{example}

\begin{proof}

As the ring $R$ has minimal multiplicity, the psuedo-Frobenius numbers of its corresponding numerical semigroup are $5,10,11$, and it follows that a fractional ideal representation for the canonical module of $R$ is $(1,t,t^6)$ by e.g. \cite[Exercise 2.14]{rosales2009numerical}. But as $t^8(1,t,t^6)=(t^8,t^9,t^{14}) \subseteq R$, it follows that $I:=(t^8,t^9,t^{14}$ is a canonical ideal for $R$. We observe that $t^{15} \cdot t^9=(t^4)^6$ and $t^{15} \cdot t^{14}=t^9 \cdot (t^4)^5$, so $t^{15} \in (t^8:I)$. By Proposotion \ref{traceprop} (3), it follows that $t^{15} \in \tr_R(I)$. On the other hand, $t^4 \cdot t^9=t^{13}=t^5 \cdot t^8$, which as $t^5 \notin R$, shows that $t^4 \notin (t^8:I)$. It follows from Theorem \ref{nearlygormainthm} that $t^4 \notin \tr_R(I)$. So $(t^8,t^9,t^{14},t^{15} \subseteq \tr_R(I) \subsetneq \m$, but as $R/(t^8,t^9,t^{14},t^{15})$ is a principal ideal ring of length $2$, it can only be that $\tr_R(I)=(t^8,t^9,t^{14},t^{15})$. 

Let $x_1:=t^8$, $x_2:=t^9$, and $x_3:={14}$ and take $F \twoheadrightarrow I$ be the graded free module surjecting onto $I$ by $e_1 \mapsto x_1$, $e_2 \mapsto x_2$, and $e_3 \mapsto x_3$. Let $z=z_1e_1+z_2e_2+z_3e_3$ be a homogeneous generator of $Z_1(I) \cap IF$, so that in particular each $z_i=t^{c_i} \in I$ for some $c_i$. Then as $z_1x_1+x_2x_2+z_3x_3=0$, it can only be that one of the $z_i$ is $0$, so as the $z_i \in I$, and as each $z_i$ is a homogeneous element of $I$, the relation $z_1x_1+z_2e_2+z_3e_3$, and in turn $z$, is determined by a relation of the form $ax_{u}x_{v}=bx_{p}x_{q}$. If either of $u$ or $v$ is equal to $p$ or $q$, then we have $z \in B_1(I)$, since either there is a common factor in the nonzero entries of $z$ which forces $z \in IZ_1(I) \subseteq B_1(I)$, or $z$ is directly a multiple of one of the homogeneous generators of $B_1(I)$. Otherwise, we must have, without loss of generality, that $u=1$, $v=3$, and $p=q=2$. But we note that $x_1x_3=t^4x_2^2$, so if $ax_1x_3=bx_2^2$, then $b=t^4a$. It follows that $z$ is a multiple of either $x_3e_1-t^4x_2e_2$ or of $x_1e_3-t^4x_2e_2$. Combining the above, it follows that $Z_1(I) \cap IF=B_1(I)+\langle x_3e_1-t^4x_2e_2,x_1e_3-t^4x_2e_2 \rangle$. But we observe that $t^4(x_3e_1-t^4x_2e_2)=t^4x_3e_1-t^8x_2e_2=x_2^2e_1-x_1x_2e_2=x_2(x_2e_1-x_1e_2) \in B_1(I)$, and $t^4(x_1e_3-t^4x_2e_2)=t^4x_1e_3-t^8x_2e_2=t^4x_1e_3-x_1x_2e_2=x_1(t^4e_3-x_2e_2) \in IZ_1(I) \subseteq B_1(I)$. It follows that $t^4\delta_1(I)=0$, but $\tr_R(I)\delta_1(I)=0$ combining \cite[Corollary 2.4]{DK21} and Proposition \ref{thetorthing}. As $(t^4)+\tr_R(I)=\m$, the claim follows.
\end{proof}


\begin{remark}
The choice of base field $\mathbb{F}_2$ in Example \ref{mkillssyztype3} is merely for simplicity. A similar argument as that of Example \ref{mkillssyztype3} can applied over any base field $k$, but comes at the expense of requiring a good deal more technical analysis to examine generators of $Z_1(I) \cap IF$. 
\end{remark}

\section{A Counterexample}\label{counterex}

In this section we provide an extended counterexample to show that Question \ref{daoetal} has a negative answer in general, even for a numerical semigroup ring with some nice properties. 


\setcounter{MaxMatrixCols}{55}


\begin{lem}\label{calcpres}

Let $R:=k[\![t^5,t^6,t^{13},t^{14}]\!]$. Then 
\[R \cong k[\![x,y,z,w]\!]/(y^3-xz,yz-xw,x^4-yw,z^2-y^2w,x^3y^2-zw,x^3z-w^2).\]

\end{lem}

\begin{proof}

Let $S:=k[\![x,y,z,w]\!]$, set $J=(y^3-xz,yz-xw,x^4-yw,z^2-y^2w,x^3y^2-zw,x^3z-w^2)$ and define the surjection $p:S \to R$ by $x \mapsto t^5$, $y \mapsto t^6$, $z \mapsto t^{13}$, and $w \mapsto t^{14}$. We observe directly that each minimal generator of $J$ is contained in $\ker(p)$, so setting $T:=S/J$, there is an induced short exact sequence of $T$-modules
\[0 \rightarrow K \rightarrow T \xrightarrow{\bar{p}} R \rightarrow 0.\]
Applying $- \otimes_T T/xT$, noting that $x$ is regular on $R$, and noting that the $T$-module structure of $R$ forces $xR=t^5R$, we have a short exact sequence 
\[0 \rightarrow K \otimes_T T/xT \rightarrow T/xT \xrightarrow{\bar{p}} R/t^5R \rightarrow 0.\]
As $t^5$ is a minimal reduction of the maximal ideal of $R$, we have $l_R(R/t^5R)=e(R)=5$. We have $T/xT \cong k[\![y,z,w]\!]/(y^3,yz,yw,z^2,zw,w^2)$. Let $\mathfrak{n}$ be the maximal ideal of $T/xT$. Then we observe that $\mathfrak{n}^3=0$ and that $\mathfrak{n}^2=(y^2)$. Then the Hilbert series of $\gr_{\mathfrak{n}}(T/xT)$ is $1+3t+t^2$, and it follows that $l_T(T/xT)=5$. But this forces $l_T(K \otimes_R T/xT)=0$, from which it follows that $K=0$, completing the proof.
\end{proof}

\begin{prop}\label{examplesyz}

Let $R:=k[\![x,y,z,w]\!]/(y^3-xz,yz-xw,x^4-yw,z^2-y^2w,x^3y^2-zw,x^3z-w^2) \cong k[\![t^5,t^6,t^{13},t^{14}]\!]$ with $\m=(x,y,z,w)$, let $L=(x,y^2,z,w)$, and consider the matrix \[S:=\begin{pmatrix} y & 0 & z & w & 0 & x^3 & 0 & 0 \\ -x & y & 0 & 0 & z & 0 & w & x^3 \\ 0 & -x & -xy & -y^2 & -y^2 & -z & -z & -w \end{pmatrix}\] and the matrix $T$
\ \\

{\arraycolsep=1.4pt $\left(\!\begin{array}{cccccccccccccccccccccccc}
\vphantom{\left\{16\right\}}y^{2}&0&z&0&0&0&w&0&0&0&x^{3}&0&0&0&0&0&0&0&0&0&0&0&0&0\\
\vphantom{\left\{17\right\}}x\,y&-y^{2}&y^{2}&0&-z&-z&z&0&-w&-w&w&x^{3}&0&-y\,w&0&0&0&0&0&0&x^{2}y^{2}&0&0&x^{2}z\\
\vphantom{\left\{23\right\}}-x&0&0&-y&0&0&0&0&0&0&0&0&-z&0&-w&-w&0&0&x^{3}&0&0&0&x^{2}y&0\\
\vphantom{\left\{24\right\}}0&0&-x&x&0&0&-y&y&0&0&0&0&0&0&z&0&0&0&-w&-w&0&0&0&0\\
\vphantom{\left\{24\right\}}0&x&0&0&0&y&0&0&0&0&0&0&0&z&0&0&-w&-w&0&0&0&0&0&0\\
\vphantom{\left\{25\right\}}0&0&0&0&0&0&0&-x&0&0&-y&0&x\,y&0&0&y^{2}&0&0&0&z&0&0&-w&0\\
\vphantom{\left\{25\right\}}0&0&0&0&x&0&0&0&y&0&0&0&0&0&0&0&z&0&0&0&-w&-w&0&0\\
\vphantom{\left\{26\right\}}0&0&0&0&0&0&0&0&0&x&0&-y&0&0&0&0&0&y^{2}&0&0&0&z&0&-w
\end{array}\!\right)$}
\ \\



 Then we have the following:

\begin{enumerate}

\item[$(1)$] $I=(x^2,xy,y^2)$ is a canonical ideal for $R$.


\item[$(2)$] There is a minimal presentation $R^{\oplus 8} \xrightarrow{S} R^{\oplus 3} \xrightarrow{p} I \rightarrow 0$ where $p$ acts on standard basis vectors by $p(e_1)=x^2$, $p(e_2)=xy$, and $p(e_3)=y^2$.

\item[$(3)$] There is a minimal presentation $R^{\oplus 24} \xrightarrow{T} R^{\oplus 8} \xrightarrow{s} \Omega^1_R(I) \rightarrow 0$ where $s$ acts on standard basis vectors by mapping $e_i$ to the $i$th column of $S$. 

\item[$(4)$] $\Omega^1_R(I) \cong L^{\oplus 2}$. 

\item[$(5)$] $\tr_R(L)=\m$.

\end{enumerate}

\end{prop}

\begin{proof}

We note that the type of $R$ is at most $3$; indeed, it is at most $4$ by e.g. \cite[Corollary 2.23]{rosales2009numerical} and cannot equal $4$ as $R$ does not have minimal multiplicity \cite[Corollary 3.2]{rosales2009numerical}. But we also note that $7,8,9$ are psuedo-Frobenius numbers of the numerical semigroup $\langle 5,6,13,14 \rangle$. As $r(R) \le 3$, we must have $r(R)=3$ from e.g. \cite[Exercise 2.14]{rosales2009numerical}, and it follows moreover that a fractional ideal representation of the canonical module of $R$ is $(1,t,t^2)$. But we observe $t^{10}(1,t,t^2)=(t^{10},t^{11},t^{12}) \subseteq R$ and that $I$ is the preimage of $(t^{10},t^{11},t^{12})$ under the isomorphism $\bar{p}$ of Lemma \ref{calcpres}. Thus $I$ is a canonical ideal for $R$, and we have Claim $(1)$.

Now, for $(2)$, set $W:=\coker(S)$. It's easy to see that $\im(S) \subseteq \ker(p)$, so there is an induced surjection $q:W \to I$. Since $x$ is regular on $I$, applying $- \otimes_R R/xR$ to the exact sequence \[0 \rightarrow \ker(q) \rightarrow W \xrightarrow{q} I \rightarrow 0,\] we obtain an exact sequence \[0 \rightarrow \ker(q) \otimes_R R/xR \rightarrow W/xW \xrightarrow{\bar{q}} I/xI \rightarrow 0\]
Let $\bar{(-)}=- \otimes_R R/xR$ and take $\mathfrak{n}$ to be the maximal ideal of $\bar{R}$.

Applying $\bar{(-)}$ to the presentation $R^{\oplus 8} \xrightarrow{S} R^{\oplus 3} \rightarrow W \rightarrow 0$, we have a presentation
\[\bar{R}^{\oplus 8} \xrightarrow{\bar{S}} \bar{R}^{\oplus 3} \rightarrow \bar{W} \rightarrow 0\]
where $\bar{S}=\begin{pmatrix} y & 0 & z & w & 0 & 0 & 0 & 0 \\ 0 & y & 0 & 0 & z & 0 & w & 0 \\ 0 & 0 & 0 & -y^2 & -y^2 & -z & -z & -w \end{pmatrix}$.
We have $\bar{R} \cong k[\![y,z,w]\!]/(y^3,yz,yw,z^2,zw,w^2)$ and we observe that $\Soc(\bar{R})=(y^2,z,w)$. Moreover, as $z,w$ are socle generators of $\mathfrak{n}$, it follows that $\mathfrak{n} \cong (y) \oplus (z) \oplus (w)$. 
As $l_R(\mathfrak{n})=l_R(\bar{R})-1=e(R)-1=4$, we have $l_R((y))=2$. But then we observe that $\begin{bmatrix} z \\ 0 \\ 0 \end{bmatrix}, \begin{bmatrix} w \\ 0 \\ -y^2 \end{bmatrix}, \begin{bmatrix} 0 \\ z \\ -y^2 \end{bmatrix}, \begin{bmatrix} 0 \\ 0 \\ -z \end{bmatrix}, \begin{bmatrix} 0 \\ w \\ -z \end{bmatrix}$, and $\begin{bmatrix} 0 \\ 0 \\ -w \end{bmatrix}$ are all socle generators of $\im(\bar{S})$. It follows that $\im(\bar{S}) \cong \im\left(\begin{bmatrix} y & 0 \\ 0 & y \\ 0 & 0 \end{bmatrix}\right) \oplus k^{\oplus 6} \cong (y)^{\oplus 2} \oplus k^{\oplus 6}$. In particular, $l_R(\im(\bar{S})=2 \cdot 2+6=10$. But as $l_R(\bar{R})=5$, additivity of length forces $l_R(\bar{W})=l_R(\bar{I})=5$. But then $l_R(\bar{\ker(q)})=0$, which forces $q$ to be an isomorphism. Claim $(2)$ now follows. 

We follow a similar approach for Claim $(3)$. Set $V:\coker(T)$. It's elementary to check that each of the 24 generators of $\im(T)$ is contained in $\ker(s)$. There is thus an induced surjection $a:V \to \Omega^1_R(I)$. Applying $\bar{(-)}$, and noting that $x$ is regular on $\Omega^1_R(I)$, we have an exact sequence
\[0 \rightarrow \bar{\ker(a)} \rightarrow \bar{V} \xrightarrow{\bar{a}} \bar{\Omega^1_R(I)} \rightarrow 0.\]
Note from above or from additivity of multiplicity that $l_R(\bar{\Omega^1_R(I)})=10$. We also have a presentation
\[\bar{R}^{\oplus 24} \xrightarrow{\bar{T}} \bar{R}^{\oplus 8} \xrightarrow{\bar{s}} \bar{V} \rightarrow 0\]
with $\bar{T}$ the matrix 
\ \\

{\arraycolsep=1.5pt $\left(\!\begin{array}{cccccccccccccccccccccccc}
\vphantom{\left\{16\right\}}y^{2}&0&z&0&0&0&w&0&0&0&0&0&0&0&0&0&0&0&0&0&0&0&0&0\\
\vphantom{\left\{17\right\}}0&-y^{2}&y^{2}&0&-z&-z&z&0&-w&-w&w&0&0&0&0&0&0&0&0&0&0&0&0&0\\
\vphantom{\left\{23\right\}}0&0&0&-y&0&0&0&0&0&0&0&0&-z&0&-w&-w&0&0&0&0&0&0&0&0\\
\vphantom{\left\{24\right\}}0&0&0&0&0&0&-y&y&0&0&0&0&0&0&z&0&0&0&-w&-w&0&0&0&0\\
\vphantom{\left\{24\right\}}0&0&0&0&0&y&0&0&0&0&0&0&0&z&0&0&-w&-w&0&0&0&0&0&0\\
\vphantom{\left\{25\right\}}0&0&0&0&0&0&0&0&0&0&-y&0&0&0&0&y^{2}&0&0&0&z&0&0&-w&0\\
\vphantom{\left\{25\right\}}0&0&0&0&0&0&0&0&y&0&0&0&0&0&0&0&z&0&0&0&-w&-w&0&0\\
\vphantom{\left\{26\right\}}0&0&0&0&0&0&0&0&0&0&0&-y&0&0&0&0&0&y^{2}&0&0&0&z&0&-w
\end{array}\!\right).$}

\ \\

Applying elementary column operations, $\bar{V}$ is isomorphic to the cokernel of the matrix 
\ \\

 \[\left(\!\begin{array}{cccccccccccccccccccccccc}
\vphantom{\left\{16\right\}}y^{2}&0&z&0&0&0&w&0&0&0&0&0&0&0&0&0&0&0&0&0&0&0&0&0\\
\vphantom{\left\{17\right\}}0&y^{2}&0&0&z&0&0&0&0&w&0&0&0&0&0&0&0&0&0&0&0&0&0&0\\
\vphantom{\left\{23\right\}}0&0&0&y&0&0&0&0&0&0&0&0&z&0&0&w&0&0&0&0&0&0&0&0\\
\vphantom{\left\{24\right\}}0&0&0&0&0&0&0&y&0&0&0&0&0&0&z&0&0&0&w&0&0&0&0&0\\
\vphantom{\left\{24\right\}}0&0&0&0&0&y&0&0&0&0&0&0&0&z&0&0&0&w&0&0&0&0&0&0\\
\vphantom{\left\{25\right\}}0&0&0&0&0&0&0&0&0&0&y&0&0&0&0&0&0&0&0&z&0&0&w&0\\
\vphantom{\left\{25\right\}}0&0&0&0&0&0&0&0&y&0&0&0&0&0&0&0&z&0&0&0&w&0&0&0\\
\vphantom{\left\{26\right\}}0&0&0&0&0&0&0&0&0&0&0&y&0&0&0&0&0&0&0&0&0&z&0&w
\end{array}\!\right).\]
Every column of this matrix corresponds to a socle generator of it's image except the columns 
\[\begin{bmatrix} 0 \\ 0 \\ y \\ 0 \\ 0 \\ 0 \\ 0 \\ 0 \\ 0 \\ 0 \\ 0 \end{bmatrix}, \begin{bmatrix} 0 \\ 0 \\ 0 \\ 0 \\ y \\ 0 \\ 0 \\ 0 \\ 0 \\ 0 \\ 0 \end{bmatrix}, \begin{bmatrix} 0 \\ 0 \\ 0 \\ y \\ 0 \\ 0 \\ 0 \\ 0 \\ 0 \\ 0 \\ 0 \end{bmatrix}, \begin{bmatrix} 0 \\ 0 \\ 0 \\ 0 \\ 0 \\ 0 \\ 0 \\ 0 \\ 0 \\ y \\ 0 \end{bmatrix}, \begin{bmatrix} 0 \\ 0 \\ 0 \\ 0 \\ 0 \\ 0 \\ 0 \\ 0 \\ y \\ 0 \\ 0 \end{bmatrix}, \begin{bmatrix} 0 \\ 0 \\ 0 \\ 0 \\ 0 \\ 0 \\ 0 \\ 0 \\ 0 \\ 0 \\ y \end{bmatrix}.\] From this, we observe that $\im(\bar{V}) \cong (y)^{\oplus 6} \oplus k^{\oplus 18}$. In particular, $l_R(\im(\bar{V}))=2 \cdot 6+18=30$. But then by additivity of length, we see that $l_R(\bar{V})=5 \cdot 8-30=10.$ As $l_R(\bar{V})=l_R(\bar{\Omega^1_R(I)})$, it follows that $\bar{\ker(a)}=0$ so that $\ker(a)=0$ as well. Thus $a$ is an isomorphism and Claim $(3)$ follows.

For $(4)$, we define the map $\theta:R^{\oplus 8} \to L^{\oplus 2}$ by
\[e_1 \mapsto (0,x) \ \ \ e_2 \mapsto (x,0) \ \ \ e_3 \mapsto (xy,y^2) \ \ \ e_4 \mapsto (y^2,z)\]
\[e_5 \mapsto (y^2,0) \ \ \ e_6 \mapsto (z,w) \ \ \ e_7 \mapsto (z,0) \ \ \ e_8 \mapsto (w,0)\]

It's easily seen that $\theta$ is surjective, and it's elementary to check that each of the 24 generators of $\im(T)$ lie in $\ker(\theta)$. So we have an induced surjection $\epsilon:V \to L^{\oplus 2}$. We note $\rank(V)=2\rank(L)=2$, so $\rank(\ker(\epsilon))=0$, but $V \cong \Omega^1_R(I)$ is torsion-free, so it must be that $\ker(\epsilon)=0$, and Claim $(4)$ follows.

For $(5)$, as $L \subseteq \tr_R(L)$, it suffices to show $y \in \tr_R(L)$. We observe that $y \cdot y^2=y^3=xz$, $y \cdot z=yz=xw$, and $y \cdot w=yw=x^4$. It follows that $y \in (x:L)$. But $(x:L) \subseteq \tr_R(L)$ by Proposition \ref{traceprop} (3), establishing $(5)$ and completing the proof.

\end{proof}

\begin{theorem}\label{finex}

Keeping the notation of Proposition \ref{examplesyz}, we have $\m \Ext^i_R(I,R)=0$ for all $i>0$, but $\tr_R(I)=\mathfrak{c}_R=(x^2,xy,y^2,z,w)$. In particular, $R$ is not nearly Gorenstein.

\end{theorem}

\begin{proof}

We note first that the ring $k[\![t^5,t^6,t^{13},t^{14}]\!]$ is a far-flung Gorenstein ring by \cite[Theorem 6.4]{He23}, that is to say the trace of its canonical ideal $(t^{10},t^{11},t^{12})$ is the conductor $(t^{10},t^{11},t^{12},t^{13},t^{14})$. That $\tr_R(I)=(x^2,xy,y^2,z,w)$ follows from noting that this ideal is the preimage of $(t^{10},t^{11},t^{12},t^{13},t^{14})$ under the isomorphism $\bar{p}$ from Lemma \ref{calcpres}, while $I$ is that of $(t^{10},t^{11},t^{12})$. 

For the claim about $\Ext$'s, we first handle the case where $i \ge 2$. Then \[\Ext^i_R(I,R) \cong \Ext^{i-1}_R(\Omega^1_R(I),R) \cong \Ext^{i-1}_R(L,R)^{\oplus 2}.\] As $\tr_R(L)=\m$, combining \cite[Lemma 2.2]{DK21} and Proposition \ref{traceprop} (4) gives that $\m \Ext^{i-1}_R(L,R)=0$ for all $i \ge 2$, so $\m \Ext^i_R(I,R)=0$ for all $i \ge 2$. 

When $i=1$, we observe from Proposition \ref{examplesyz} that 
\[R^{\oplus 24} \xrightarrow{T} R^{\oplus 8} \xrightarrow{S} R^{\oplus 3} \rightarrow 0\]
is the beginning of a minimal free resolution of $I$. Then $\Ext^1_R(I,R)$ is first homology of the complex
\[0 \rightarrow R^{\oplus 3} \xrightarrow{S^T} R^{\oplus 8} \xrightarrow{T^T} R^{\oplus 24},\]
i.e., $\Ext^1_R(I,R) \cong \ker(T^T)/\im(S^T)$. We claim $\ker(T^T)=\im(P)$ where
\[P:=\left(\!\begin{array}{cccccccc}
\vphantom{\left\{-16\right\}}0&x&-x&y&-y^{2}&z&-z&w\\
\vphantom{\left\{-17\right\}}x&0&y&0&z&0&w&0\\
\vphantom{\left\{-23\right\}}x\,y&y^{2}&0&z&0&y\,w&0&x^{3}y\\
\vphantom{\left\{-24\right\}}y^{2}&z&0&w&0&x^{3}y&0&x^{2}y^{2}\\
\vphantom{\left\{-24\right\}}y^{2}&0&z&0&y\,w&0&x^{3}y&0\\
\vphantom{\left\{-25\right\}}z&w&0&x^{3}&0&x^{2}y^{2}&0&x^{2}z\\
\vphantom{\left\{-25\right\}}z&0&w&0&x^{3}y&0&x^{2}y^{2}&0\\
\vphantom{\left\{-26\right\}}w&0&x^{3}&0&x^{2}y^{2}&0&x^{2}z&0
\end{array}\!\right).\]
To see this, we follow as similar approach as in the proof of Theorem \ref{examplesyz}, and note first it is elementary to check that $\im(P) \subseteq \ker(T^T)$. So setting $U:=\coker(P)$ and $V:=R^{\oplus 8}/\ker(T^T) \cong \im(T^T)$, we have a natural surjection $q:U \twoheadrightarrow V$. Set $\bar{R}:=- \otimes_R R/xR$. As $V \hookrightarrow R^{\oplus 24}$, $V$ is torsion-free, so applying $\bar{(-)}$ to the exact sequence 
\[0 \rightarrow \ker(q) \rightarrow U \xrightarrow{q} V \rightarrow 0,\]
we get an exact sequence
\[0 \rightarrow \bar{\ker(q)} \rightarrow \bar{U} \xrightarrow{\bar{q}} \bar{V} \rightarrow 0.\]

Noting that $R/xR \cong k[\![y,z,w]\!]/(y^3,yz,yw,z^2,zw,w^2)$ we have a minimal presentation
\[\bar{R}^{\oplus 8} \xrightarrow{\bar{P}} \bar{R}^{\oplus 8} \rightarrow \bar{U} \rightarrow 0\]
where 
\[\bar{P}=\left(\!\begin{array}{cccccccc}
\vphantom{\left\{-16\right\}}0&0&0&y&-y^{2}&z&-z&w\\
\vphantom{\left\{-17\right\}}0&0&y&0&z&0&w&0\\
\vphantom{\left\{-23\right\}}0&y^{2}&0&z&0&0&0&0\\
\vphantom{\left\{-24\right\}}y^{2}&z&0&w&0&0&0&0\\
\vphantom{\left\{-24\right\}}y^{2}&0&z&0&0&0&0&0\\
\vphantom{\left\{-25\right\}}z&w&0&0&0&0&0&0\\
\vphantom{\left\{-25\right\}}z&0&w&0&0&0&0&0\\
\vphantom{\left\{-26\right\}}w&0&0&0&0&0&0&0
\end{array}\!\right).\]

As in the proof of Theorem \ref{examplesyz}, we have $\Soc(\bar{R})=(y^2,z,w)$, that $l_{\bar{R}}((y))=2$, and we also observe that $\Ann_{\bar{R}}((y))=\Soc(\bar{R})$ so that $(y) \cong \bar{R}/\Soc(\bar{R})$. Then every column of $\bar{P}$ is a socle generator of $\im(\bar{P})$ except the columns $\begin{bmatrix} 0 \\ y \\ 0 \\ 0 \\ z \\ 0 \\ w \\ 0 \end{bmatrix}$ and $\begin{bmatrix} y \\ 0 \\ z \\ w \\ 0 \\ 0 \\ 0 \\ 0 \end{bmatrix}$. So 
\[\im(\bar{P}) \cong \im\left(\begin{bmatrix} 0 & y \\ y & 0 \\ 0 & z\\ 0 & w\\ z & 0 \\ 0 & 0 \\ w & 0\\ 0 & 0\end{bmatrix}\right) \oplus k^{\oplus 6} \cong (\bar{R}/\Soc(\bar{R}))^{\oplus 2} \oplus k^{\oplus 6} \cong (y)^{\oplus 2} \oplus k^{\oplus 6}.\]

In particular, we have $l_R(\im(\bar{P}))=10$, so noting that $l_R(\bar{R})=5$, additivity of length gives that $l_R(\bar{U})=8 \cdot 5-10=30$. 

Now since $\Ext^1_R(I,R)$ has finite length, we have $\rank(\ker(T^T))=\rank(\im(S^T))=3-\rank(I^*)=2$. Then $\rank(V)=8-\rank(\ker(T^T))=6$. So $e_R(V)=e(R) \cdot 6=30$. But $V \hookrightarrow R^{\oplus 24}$, and so is torsion-free. As $x$ is a redution of $\m$, $l_R(\bar{V})=e_R(V)=30$. But then $\bar{q}$ is a surjection between two modules of the same length, so is an isomorphism. It follows that $\bar{\ker(q)}=0$ so that $\ker(q)=0$ as well, i.e. that $q$ is an isomorphism. But $q$ is given on basis vectors by $e_i \mapsto e_i$, so we have $\ker(T^T)=\im(P)$, as desired. 

It is now elementary to check explicitly that $x\ker(T^T) \subseteq \im(S^T)$ and $y\ker(T^T) \subseteq \im(S^T)$ 
As $z,w \in \tr_R(I)$, it follows from \cite[Corollary 2.4]{DK21} that $z,w \in \Ann_R(\Ext^1_R(I,R))$. Thus $\m \Ext^1_R(I,R)=0$, completing the proof.

\end{proof}

\begin{remark}

The argument actually shows that for $R=k[\![t^5,t^6,t^{13},t^{14}]\!]$ we have $\m \Ext^i_R(I,-)=0$ for all $i \ge 2$. It is only $\Ext^1_R(I,-)$ that, in view of \cite[Theorem 2.3]{DK21}, provides an obstruction to $\tr_R(I)$ being equal to $\m$.

\end{remark}

\begin{remark}
The example $R=k[\![t^5,t^6,t^{13},t^{14}]\!]$ giving a negative answer to Question \ref{daoetal} is minimal in several regards. For instance $R$ has almost minimal multiplicity whereas Theorem \ref{nearlygormainthm} implies Question \ref{daoetal} has an affirmative answer for numerical semigroup rings of minimal multiplicity. The ring $R$ has type $3$, while \cite[Theorem 4.4]{DK21} implies an affirmative answer for numerical semigroup rings of type $2$. $R$ has embedding dimension $4$ and multiplicity $5$, while a numerical semigroup ring of embedding dimension at most $3$ or multiplicity at most $4$ must be either Gorenstein, have type $2$, or have minimal multiplicity, reducing again to the results mentioned above. So in each of these respects the example $R$ is minimal.
\end{remark}

\section*{Acknowledgements}

This material is based upon work supported by the National Science Foundation under Grant No. DMS-1928930 and by the Alfred P. Sloan Foundation under grant G-2021-16778, while the first author was in residence at the Simons Laufer Mathematical Sciences Institute (formerly MSRI) in Berkeley, California, during the Spring 2024 semester. The second author was supported partially by Project No. 51006801 - American Mathematical Society-Simons Travel Grant.

\bibliographystyle{alpha}
\bibliography{mybib}

\vspace{.3cm}

\end{document}